\documentclass[a4paper,11pt]{article}

\newcommand{\beq}{\begin{equation}}
\newcommand{\eeq}{\end{equation}}

\newcommand{\comment}[1]{}

\usepackage{natbib}
\usepackage{enumitem}
\usepackage{tikz}

\usepackage{amsmath,amssymb,amsthm}
\usepackage{mathrsfs} 
\usepackage{dsfont} 

\numberwithin{equation}{section}

\theoremstyle{plain}
\newtheorem{thm}{Theorem}[section]
\newtheorem{prop}[thm]{Proposition}
\newtheorem{cor}[thm]{Corollary}
\newtheorem{lem}[thm]{Lemma}
\newtheorem{defn}[thm]{Definition}
\newtheorem{asm}{Assumption}
\theoremstyle{remark}
\newtheorem{rem}[thm]{Remark}
\newtheoremstyle{key}{3pt}{3pt}{}{}{\itshape}{:}{.5em}{}
\theoremstyle{key}
\newtheorem*{MSC}{MSC2010 subject classification}
\newtheorem*{JEL}{JEL subject classification}
\newtheorem*{keywords}{Keywords}

\DeclareMathOperator*{\esssup}{ess\,sup}

\newcommand{\nbd}[1]{\(#1\)\nobreakdash-\hspace{0pt}}

\newcommand{\indi}[1]{\mathbf{1}_{#1}}

\newcommand{\R}{\mathds{R}}

\newcommand{\Exp}{E}
\newcommand{\Pb}{P}

\newcommand{\F}{{\mathscr F}}

\usepackage{url}

\begin{document}
\title{\textbf{Continuous-Time Public Good Contribution \\ under Uncertainty: A Stochastic Control Approach}\footnote{Financial support by the German Research Foundation (DFG) via grant Ri 1142-4-2 is gratefully acknowledged.}}
\author{Giorgio Ferrari \thanks{Center for Mathematical Economics, Bielefeld University, Germany; \texttt{giorgio.ferrari@uni-bielefeld.de}}
\and Frank Riedel \thanks{Center for Mathematical Economics, Bielefeld University, Germany, and Department of Economic and Financial Services, University of Johannesburg, Republic of South Africa; \texttt{frank.riedel@uni-bielefeld.de}} \and Jan-Henrik Steg \thanks{Center for Mathematical Economics, Bielefeld University, Germany; \texttt{jsteg@uni-bielefeld.de}}}
\date{\today}
\maketitle

\vspace{0.5cm}

{\textbf{Abstract.}} In this paper we study continuous-time stochastic control problems with both monotone and classical controls motivated by the so-called public good contribution problem. That is the problem of $n$ economic agents aiming to maximize their expected utility allocating initial wealth over a given time period between private consumption and irreversible contributions to increase the level of some public good. 
We investigate the corresponding social planner problem and the case of strategic interaction between the agents, i.e.\ the public good contribution game.
We show existence and uniqueness of the social planner's optimal policy, we characterize it by necessary and sufficient stochastic Kuhn-Tucker conditions and we provide its expression in terms of the unique optional solution of a stochastic backward equation. Similar stochastic first order conditions prove to be very useful for studying any Nash equilibria of the public good contribution game. In the symmetric case they allow us to prove (qualitative) uniqueness of the Nash equilibrium, which we again construct as the unique optional solution of a stochastic backward equation. We finally also provide a detailed analysis of the so-called free rider effect.

\maketitle

\vspace{0.5cm}
\begin{MSC}
93E20, 91B70, 91A15, 91A25, 60G51
\end{MSC}

\begin{JEL}
C02, C61, C62, C73
\end{JEL}

\begin{keywords}
singular stochastic control, stochastic games, first order conditions for optimality, Nash equilibrium, L\'evy processes, irreversible investment, public good contribution, free-riding. 
\end{keywords}

\section{Introduction}\label{sec:intro}

In economics, a private good is \emph{excludable}, i.e.\ its owners can exercise private property rights, and \emph{rivalrous}, i.e.\ consumption by one necessarily prevents that by another. Private goods are, for example, food, consumable resources (like energy or tap water), but also most durables. On the other hand, public goods are \emph{nonexcludable} and \emph{nonrivalrous}. Clean environment, national security, academic research and accessible public capital like infrastructure are well known examples of public goods.
 
In this paper we consider the problem of $n$ economic agents aiming at maximizing their expected utility by allocating their initial wealth over a given time period $[0,T]$, $T \in (0,\infty]$, between consumption of a private good and irreversible investment to increase the aggregate level of some public good. That is, a so-called \emph{public good contribution problem}.
Firstly, we investigate the corresponding social planner problem, in which a fictitious decision maker takes care of the consumption choices of all the agents so to maximize the social welfare. Secondly, we consider the case of strategic interaction between the agents, i.e.\ a public good contribution game.
We model these classical economic problems in terms of general stochastic control problems in continuous time with both monotone (possibly singular with respect to the Lebesgue measure, as a function of time) and classical absolutely continuous control processes. The monotone controls represent the cumulative contributions into the public good, which are assumed irreversible, whereas the instantaneous consumption of the private good is modeled as a nonnegative rate with respect to time. 

We tackle the control problem by a first order condition approach that may be thought of as a stochastic, infinite-dimensional generalization of the classical Kuhn-Tucker conditions. Our method does not require any Markovian or diffusive hypothesis, and in this sense it represents a substitute in non-Markovian frameworks for the Hamilton-Jacobi-Bellman equation. This approach is very powerful in solving general singular control problems as it has been shown in a quite recent literature. We refer to \citet{BankRiedel01,BankRiedel03} for an intertemporal utility maximization problem with Hindy, Huang and Kreps preferences; to \cite{Bank05}, \cite{ChiarollaFerrari11}, \cite{Ferrari12} and \cite{RiedelSu11} for the irreversible investment problem of a monopolistic firm with both limited and unlimited resources; to \cite{Chiarollaetal12} for the social planner problem in a market with $N$ firms and limited resources; to \cite{Steg12} for a capital accumulation game.

We start analyzing the public good contribution problem by taking the point of view of a social planner who aims to maximize the expected total utility of the economy under a social budget constraint. Assuming (forward) prices of the public good and of the private consumption given by discounted exponential martingales, we prove existence and uniqueness of the social planner's optimal policy via a suitable application of Koml\'os' classical theorem \citep[cf.][]{Komlos67}, resp.\ of a generalization due to Kabanov (cf.\ Lemma 3.5 in \citealp{Kabanov99}; see also the functional analytic version of \citealp{Balder}). The optimal investment strategy is completely characterized by a set of necessary and sufficient stochastic Kuhn-Tucker conditions, which in turn lead to the identification of a signal process that triggers the optimal monotone control, i.e.\ the optimal investment rule into the public good. Such a signal process is the unique solution of a backward stochastic equation in the spirit of \cite{BankElKaroui04}. 

We then consider strategic interaction between the agents in our economy and we show that any Nash equilibrium (if it does exist) satisfies a similar set of first order conditions for optimality. In equilibrium, the agents act optimally, taking the contribution \emph{processes} of others as given.
In this sense we restrict our attention to \textsl{open-loop} strategies (see also \citealp{BackPaulsen09}, and \citealp{Steg12}) without explicit reactions to \emph{deviations} from announced (equilibrium) play. Indeed, as pointed out by \cite{BackPaulsen09}, there are serious conceptual problems defining a stochastic continuous-time game of singular controls as ours with more explicit feedback (\textsl{closed-loop}) strategies.
In a symmetric setting in which agents have utility functions of the same form and the same initial wealth, we show that there exists a symmetric Nash equilibrium, in which agents at any time consume the same amount of private good and invest the same amount in the public good. Further, any equilibrium has the same aggregate public good process as the symmetric one and the same private good consumption, which is hence a qualitative uniqueness result. As for the social planner solution, we are able to completely characterize the equilibrium public good level in terms of the running supremum of an optional process uniquely solving a backward stochastic equation \`a la \cite{BankElKaroui04}.

Public goods are extensively studied in economics because of the so-called \emph{free rider effect}: agents enjoy the contributions of others but do not take into account others' benefits when making their own contributions (see, e.g., \citealp{CornesSandler96}, or \citealp{Laffont88}). This leads to an inefficiently low voluntary (strategic) supply of the public good. We confirm this phenomenon in the symmetric version of our general setting (with general concave utilities and general stochastic price processes): the total expenditure for the public good associated to the social planner's symmetric optimal policy is higher than the one related to a symmetric Nash equilibrium. Moreover, our approach also allows a detailed analysis of the degree of free riding in a more specific symmetric setting, where utilities are of Cobb-Douglas type and prices are driven by L\'evy uncertainty. Indeed, in such a case we can find the explicit forms of the social planner's optimal policy and of the Nash equilibrium so to compare their public good contribution processes. However, such a detailed comparison seems very hard to obtain in the general setting (see the discussion after Proposition \ref{prop:expenditure} below).

From the applied point of view, our model differs from the existing related literature\footnote{%
See, e.g., the classical work of \cite{Varian94} and the recent papers by \cite{Battaglinietal12}, \cite{Wang10} and \cite{YeungPetrosyan13} and references therein.
}
on public goods for two main aspects.
First of all, we consider a \emph{stochastic, dynamic} model of \emph{intertemporal} investment choice with general concave utilities, not necessarily separable, which thus allow us to account for cross effects between the public and the private good. To the best of our knowledge, this is a novelty with respect to the classical models in which usually the investor has a quasilinear utility\footnote{A quasilinear utility is an additively separable utility function which is linear in one of its arguments and concave in the others.} and can choose how to divide a budget given in each period between instantaneous private consumption and public good investment.
Secondly, we take into account \emph{irreversibility} of the investments (by modeling the public good contributions as monotone controls) which together with uncertainty typically induces reluctance to invest. We are able to analyze the interplay of this dynamic effect with free riding. Interestingly it is not necessarily the case that uncertainty and irreversibility of public good contributions aggravate the degree of free-riding.  We indeed explicitly evaluate the free rider effect in a symmetric Black-Scholes setting with Cobb-Douglas utilities and we show that uncertainty and irreversibility of public good provisions do not affect the \emph{degree} of free-riding.

The paper is organized as follows. In Section \ref{Model} we set up the model. In Section \ref{SocialPlanner} we consider the social planner problem, proving existence and uniqueness of its solution and introducing the stochastic Kuhn-Tucker conditions for optimality. The public good contribution game is addressed in Section \ref{game}, whereas explicit results are obtained in Section \ref{freerider}. Finally we refer to Appendix \ref{AppProofs} for some technical proofs.


\section{The Model}
\label{Model}
We consider a continuous-time stochastic economy with a finite number $n\geq 1$ of agents over a fixed time horizon $0<T \leq \infty$. Each agent, indexed by \(i=1,\dots,n\), chooses how to allocate his initial wealth $w^i>0$ between private consumption $x^i$ and arbitrary but nondecreasing cumulative contributions $C^i$ to increase the level of some public good. We assume a continuous revelation of information about an exogenous source of uncertainty and we allow the agents to condition their decisions on the accumulated information. Formally, let \((\Omega,\F,\{\F_{t}\}_{t \in [0,T]},\Pb)\) be a filtered probability space satisfying the usual conditions of right-continuity and completeness. For the moment we do not make any Markovian assumption.

One may think that the agents are financed entirely by their labor or by holding a portfolio of financial instruments. Hence they are part of a more complex financial market that, however, we do not model explicitly. At the initial time each agent can buy one unit of the private good for contingent delivery at time $t\in [0,T]$ and state $\omega \in \Omega$ at forward price $\psi_x(\omega,t)$. Analogously, the contingent state-price for the contribution to the public good is $\psi_c(\omega,t)$. Both $\psi_x$ and $\psi_c$ are strictly positive (more technical conditions are listed in Assumption \ref{asm2} below). The forward price of an investment plan $(x^i, C^i)$ is therefore
\beq
\label{budget}
\Psi(x^i,C^i):=\Exp\biggl[\int_0^T \psi_x(t) x^i(t) dt + \int_0^T  \psi_c(t) dC^i(t)\biggr],
\eeq 
where $(x^i, C^i)$ can be chosen in the nonempty, convex budget-feasible set
\begin{equation}\label{Bwi}
\begin{split}
\mathcal{B}_{w^i}:=\biggl\{&(x^i, C^i)\colon \Omega \times [0,T] \mapsto \mathbb{R}_{+}^2\, \text{ optional, s.t.\ } C^i \text{ is right-continuous} \\ 
& \text{and nondecreasing with } C^i(0-) =0 \ \Pb\text{-a.s.}, \text{ and } \Psi(x^i,C^i) \leq w^i \biggr\}. 
\end{split}
\end{equation}
\begin{rem}
\label{rem:xoptional}
The assumption that $(x^i,C^i)$ is optional is not more restrictive than adaptedness. Indeed, adaptedness and right-continuity of $C^i$ imply optionality, cf.\ footnote \ref{fn:asmpsi}. If $x^i$ is only adapted, then its optional projection ${}^ox^i$ satisfies ${}^ox^i_t=x^i_t$ \nbd{\Pb}a.s.\ for all $t\in[0,T]$ and hence yields the same cost in \eqref{budget} and utility in \eqref{payoff} below thanks to Tonelli's Theorem.
\end{rem}

Here $\Psi(x^i,C^i)\leq w^i$ defines therefore the budget constraint of agent $i$. Notice that we can have jumps as well as singular increases (with respect to the Lebesgue measure) in $C^i$. This means that agents can contribute both in lumps at certain times and continuously, the latter even not necessarily in rates. Recalling that each $t \mapsto C^i(\omega,t)$ is a.s.\ nondecreasing, we in fact denote by $\int_0^T  (\cdot)\, dC^i(t)$ the Lebesgue-Stieltjes integral $\int_{[0,T]}(\cdot)\, dC^i(t)$ throughout, to include a possible initial jump of the contribution, corresponding to a point mass $C^i(0) > 0$ of the random Borel measure $dC^i$ at time $t=0$.

The agents are assumed to derive some expected, time-separable utility from the private good and the aggregate public good process \(C:=\sum_{i\in\{1,\dots,n\}}C^{i}\). Given a combination of strategies from $\prod_{i=1}^n \mathcal{B}_{w^i}$, agent \(i\)'s utility is

\begin{equation}
\label{payoff}
U^{i}(x^i,C^{i}; C^{-i}):=\Exp\biggl[\int_0^T e^{-\int_0^t r(s)\,ds}u^{i}(x^i(t), C(t))\,dt\biggr],
\end{equation}
where \(C^{-i}:=\sum_{j\in\{1,\dots,n\}\setminus i}C^{j}\), $r$ is an exogenous stochastic discount factor and \(u^{i}:\R_{+}^2 \mapsto\R_{+}\) is an instantaneous utility function. 
In the economic literature on public good contribution it is customary to assume quasilinear utilities (see, e.g., the early paper by \citealp{Varian94}, and the very recent one by \citealp{Battaglinietal12}). Here, instead, we work with general concave utilities, allowing to account for cross effects between the public and the private good.
\begin{asm}
\label{asm1}
\indent\par
\begin{enumerate}
\item\label{Ui_x}
$u^{i}$ is increasing and strictly concave on \(\R_+^2\), as well as twice continuously differentiable on the open cone \(\R_{++}^2\), with negative definite Hessian. Moreover, it satisfies the Inada conditions
$$\lim_{x\downarrow 0} u_x^i(x, c) =  + \infty \quad \text{ and } \quad \lim_{x \uparrow \infty}u_x^i(x, c) = 0$$
for any $c>0$. 

\item\label{unint}
The family $\displaystyle \Big(e^{-\int_0^t r(\omega,s)\,ds}u^{i}(x^i(\omega,t), C(\omega,t)),\,\,\,(x,C) \in \prod_{i=1}^n \mathcal{B}_{w^i}\Big)$ is $\Pb \otimes dt$-uniformly integrable.
\end{enumerate}
\end{asm}

\begin{asm}
\label{asm2}
\indent\par
\begin{enumerate}
\item\label{psi_c}
The adapted process $\psi_c:=\{\psi_c(t), t \in [0,T]\}$ is right-continuous and such that  $\psi_c(t)=e^{-\int_0^t r_c(s)ds}\mathcal{E}_c(t)$ a.s.\ for every $t\in[0,T]$, with some uniformly bounded and strictly positive process $r_c=\{r_c(t), t \in [0,T]\}$ and some exponential martingale $\mathcal{E}_c:=\{\mathcal{E}_c(t), t \in [0,T]\}$. Moreover, if $T=\infty$, one has $\psi_c(T)=0$ a.s.
\item \label{psi_x}
The adapted process $\psi_x:=\{\psi_x(t), t \in [0,T]\}$ is right-continuous and such that $\psi_x(t)=e^{-\int_0^t r_x(s)ds}\mathcal{E}_x(t)$ a.s.\ for every $t\in[0,T]$, with some uniformly bounded and strictly positive process $r_x:=\{r_x(t), t \in [0,T]\}$ and some exponential martingale $\mathcal{E}_x:=\{\mathcal{E}_x(t), t \in [0,T]\}$. 
\item The optional, strictly positive process $r:=\{r(t), t \in [0,T]\}$ is uniformly bounded.\footnote{\label{fn:asmpsi}%
A stochastic process $X$ is: 
\begin{enumerate}
\item 
\textsl{optional} if it is measurable with respect to the optional sigma-field $\mathcal{O}$ on $\Omega\times[0,T]$ generated, e.g., by the right-continuous adapted processes;

\item 
\textsl{lower-semicontinuous in expectation} if for any stopping time $\tau$ one has $\liminf_{n\uparrow \infty}\Exp[X(\tau_n)] \geq \Exp[X(\tau)]$, whenever $\{\tau_n\}_{n\in \mathbb{N}}$ is a monotone sequence of stopping times converging to $\tau$;

\item 
of \textsl{class (D)} if $\{X(\tau),\,\,\tau\,\,\text{a stopping time}\}$ defines a uniformly integrable family of random variables on $(\Omega,\F,\Pb)$.
\end{enumerate}
We refer the reader to \cite{DellacherieMeyer78}, among others, for further details.
}  
\end{enumerate}
\end{asm}

Under Assumptions \ref{asm1} and \ref{asm2} the payoff in (\ref{payoff}) is well defined and finite for any $i=1,\dots,n$.

\pagebreak
\begin{rem}%
\label{ontheassumptions}%
\indent\par
\begin{enumerate}
	\item The Inada conditions of Assumption \ref{asm1}.i.\ guarantee that there will be an interior solution for optimal private consumption.
	\item Note that since \(u^{i}\) is concave in \(c\), by Assumption \ref{asm1}.\ref{unint}.\ \(e^{-\int_0^t r(s)\,ds}u^{i}_c(x(t),C(t))\) is $\Pb \otimes dt$-integrable for any \((x,C) \in \prod_{i=1}^n \mathcal{B}_{w^i}\) such that $C(0)>0$ a.s.
	\item Thanks to the budget constraint in \eqref{Bwi}, when $T<\infty$ one can easily adapt the arguments in the proof of Lemma 2.1 in \cite{BankRiedel01} to show that Assumption \ref{asm1}.ii.\ is satisfied if
	\begin{enumerate}
		\item for some $\alpha,\beta \in (0,1)$ one has $u^i(x,c) \leq \text{const.}(1 + x^{\alpha} + c^{\beta})$, and
	\item $$\mathcal{E}_x^{-1} \in L^{\hat{p}}(\Pb) \quad \text{and} \quad \mathcal{E}_c^{-1} \in L^{\hat{q}}(\Pb)$$  for some $\hat{p} > \frac{\alpha}{1-\alpha}$ and $\hat{q} > \frac{\beta}{1-\beta}$, with $\mathcal{E}_x$ and $\mathcal{E}_c$ as in Assumption \ref{asm2}.
	\end{enumerate}
	\item Right-continuity of $\psi_c$ follows already from taking a right-continuous version of the exponential martingale $\mathcal{E}_c$. By assuming that $\mathcal{E}_c$ has a last element, it is uniformly integrable and hence of class (D) \citep[see][Theorem 3.2]{RevuzYor}, cf.\ footnote \ref{fn:asmpsi}; the same holds then for $\psi_c$. Furthermore, $\psi_c$ is a supermartingale under Assumption \ref{asm2}.\ref{psi_c}.\ and hence lower semicontinuous in expectation also from the left. Indeed, $\Exp\bigl[\psi_c(u)\big|\F_t\bigr]\leq e^{-\int_0^t r_c(s)ds}\Exp\bigl[\mathcal{E}_c(u)\big|\F_t\bigr]=\psi_c(t)$. The same remarks apply to $\psi_x$.
	\item It is easy to see that Assumption \ref{asm2}.\ref{psi_c}.\ and \ref{asm2}.\ref{psi_x}.\ are satisfied, for example, by the classical benchmark cases of geometric Brownian motions (for discount factor $r_c$ suitably chosen to have $\psi_c(T)=0$ a.s.\ when $T=+\infty$).
	\item More generally, in Sections \ref{SocialPlanner} and \ref{game} below we can allow for a random field $u^i\colon\Omega\times[0,T]\times\R_+^2$ of instantaneous utilities such that for any given $(x,C)\in \prod_{i=1}^n \mathcal{B}_{w^i}$, the process $(\omega,t)\mapsto u^{i}(\omega,t,x(\omega,t),C(\omega,t))$ is progressively measurable. Then the properties in Assumption \ref{asm1}.i.\ are imposed for any given $(\omega,t)\in\Omega\times[0,T]$ on the mapping $(x,c)\mapsto u^{i}(\omega,t,x,c)$, while in ii.\ just the argument $(\omega,t)$ needs to be added.
\end{enumerate}
\end{rem}



\section{The Social Planner Problem}
\label{SocialPlanner}

We start our analysis by studying a social planner problem for the economy described in Section \ref{Model}. Throughout this section, denote by $(\underline{x}, \underline{C})$ a vector of investment processes valued in $\mathbb{R}_{+}^{2n}$ with components $(x^1,\dots,x^n,C^1,\dots,C^n)$ and introduce the nonempty, convex, social budget-feasible set
\begin{equation}
\label{Bw}
\begin{split}
\mathcal{B}_{w}:=\biggl\{&(\underline{x}, \underline{C})\colon \Omega \times [0,T] \mapsto \mathbb{R}_{+}^{2n}\, \text{ optional, s.t.\ } C^i \text{ is right-continuous and} \\ 
&\text{nondecreasing with } C^i(0-)=0\ \Pb\text{-a.s.},\, i=1,\dots,n, \text{ and } \sum_{i=1}^n \Psi(x^i,C^i) \leq w \biggr\}
\end{split}
\end{equation}
with $w:=\sum_{i=1}^n w^i$. We say that $(\underline{x}, \underline{C})$ is admissible if $(\underline{x}, \underline{C}) \in \mathcal{B}_{w}$. Suppose that there exists a fictitious social planner aiming to maximize the aggregate expected utility by allocating efficiently the available initial wealth. This amounts to solving the optimization problem with value 
\begin{equation}
\label{SPproblem}
V_{SP}(w):=\sup_{(\underline{x}, \underline{C}) \in \mathcal{B}_{w}} U_{SP}(\underline{x}, \underline{C}) := \sup_{(\underline{x}, \underline{C}) \in \mathcal{B}_{w}} \sum_{i=1}^n \gamma^i U^{i}(x^i,C^{i};C^{-i})
\end{equation}
with $U^{i}(x^i,C^{i};C^{-i})$ as in (\ref{payoff}) and for given positive weights $\gamma^i$, $i=1,\dots,n$, such that $\sum_{i=1}^n \gamma^i =1$.

\begin{thm}
\label{SPexistence}
Under Assumptions \ref{asm1} and \ref{asm2} there exists $(\underline{x}_{*}, \underline{C}_{*}) \in \mathcal{B}_w$ which solves the social planner problem \eqref{SPproblem}. $(\underline{x}_{*}, C_{*})$ with $C_*=\sum_{i=1}^n C^i_*$ is unique up to indistinguishability.
\end{thm}

\begin{proof}
The proof is organized in three steps.\vspace{0.15cm}

\textsl{Step 1.}\quad In this first step we let $T < \infty$. Recall that $\psi_x(t)=e^{-\int_0^t r_x(s)ds}\mathcal{E}_x(t)$ and $\psi_c(t)=e^{-\int_0^t r_c(s)ds}\mathcal{E}_c(t)$ for some strictly positive and uniformly bounded processes $r_x$ and $r_c$, and for some exponential martingales $\mathcal{E}_x$ and $\mathcal{E}_c$ (cf.\ Assumption \ref{asm2}.\ref{psi_c}.).
Let $\tilde{\Exp}_c$ and $\tilde{\Exp}_x$ be the expectations under the measures $\tilde{\Pb}_c$ and $\tilde{\Pb}_x$ on $\F_T$ with Radon-Nikodym derivative $\mathcal{E}_c(T)$ and $\mathcal{E}_x(T)$, respectively, with respect to $\Pb$. Since $\mathcal{E}_x(T)>0$ and $\mathcal{E}_c(T)>0$ a.s., the measure $\Pb$ is equivalent both to $\tilde{\Pb}_c$ and $\tilde{\Pb}_x$. Also, from now on we shall set $d\mu_x:= d\tilde{\Pb}_x \otimes dt/T$.

Let $\{(\underline{x}_m, \underline{C}_m)\}_{m \in \mathbb{N}} \subset \mathcal{B}_w$ be a maximizing sequence; that is, a sequence of investment plans such that
$$\lim_{m \rightarrow \infty} \sum_{i=1}^n \gamma^i U^{i}(x^i_m,C^{i}_m;C^{-i}_m) = V_{SP}.$$
Notice that for any $i=1,\dots,n$
\begin{eqnarray}
\label{stimax}
w & \hspace{-0.25cm}  \geq \hspace{-0.25cm}  & \Exp\biggl[\int_0^T\psi_x(t)x^i_m(t) dt\biggr] = \Exp\biggl[\int_0^T e^{-\int_0^t r_x(s)ds} \Exp[\mathcal{E}_x(T)|\F_{t}]\,x^i_m(t)dt\biggr] \nonumber \\   
& \hspace{-0.25cm} = \hspace{-0.25cm} &  \Exp\biggl[\mathcal{E}_x(T) \int_0^Te^{-\int_0^t r_x(s)ds}x^i_m(t)dt\biggr] \geq K_1 \tilde{\Exp}_x\biggl[\int_0^T x^i_m(t) dt\biggr],
\end{eqnarray}
where the uniform boundedness of $r_x$ implies the last step with some constant $K_1>0$.
Hence each $\{x^i_m\}_{m\in \mathbb{N}}$, $i=1,\dots,n$, is $L^1$-bounded as sequence of random variables on the probability space $(\Omega \times [0,T],\mathcal{O}, d\mu_x)$, where $\mathcal{O}$ is the optional $\sigma$-algebra. Analogously, 
\begin{align*}
w &\geq \Exp\biggl[\int_0^T\psi_c(t)dC^i(t)\biggr] = \Exp\biggl[\int_0^T e^{-\int_0^t r_c(s)ds} \Exp[\mathcal{E}_c(T)|\F_{t}]\,dC^i(t)\biggr]\\
&=\Exp\biggl[\mathcal{E}_c(T) \int_0^T e^{-\int_0^t r_c(s)ds} dC^i(t)\biggr] 
= \tilde{\Exp}_c \biggl[\int_0^T e^{-\int_0^t r_c(s)ds} dC^i(t)\biggr] \geq K_2\tilde{\Exp}_c [C^i(T)],
\end{align*}
where the second equality follows from Theorem $1.33$ in \cite{Jacod79}, and with $K_2 > 0$ a suitable constant implied by the uniform boundedness of $r_c$.

Hence by Koml\'os' theorem (see \citealp{Komlos67}, for its standard formulation and \citealp{Kabanov99}, Lemma 3.5, for a version of Koml\'os' theorem for optional random measures\footnote{An optional random measure on $[0,T]$ is simply a random variable $\nu$ valued in the space of nonnegative Borel measures on $[0,T]$ (endowed with the topology of weak*-convergence) such that the process $\nu(\omega,t):=\nu(\omega, [0,t])$ is adapted. Our admissible public good contribution processes are the cumulative distributions of optional random measures, being adapted, right-continuous, and nonnegative.}), for every $i=1,\dots,n$ there exist two subsequences $\{\tilde{x}^i_m\}_{m \in \mathbb{N}} \subset \{x^i_m\}_{m \in \mathbb{N}}$ and $\{\tilde{C}^i_m\}_{m \in \mathbb{N}} \subset \{C^i_m\}_{m \in \mathbb{N}}$ and a $d\mu_x$-integrable and optional process $x^i_{*}$ and some optional random measure $dC^i_{*}$, $i=1,\dots,n$, such that, as $k\uparrow \infty$,
\begin{equation}
\label{convX}
X^i_k(t):=\frac{1}{k+1}\sum_{m=0}^k \tilde{x}^i_m(t) \rightarrow x^i_{*}(t),\,\,\,d\mu_x\text{-a.e.},
\end{equation}
and
\begin{equation}
\label{KomlosI}
I^i_k(t):=\frac{1}{k+1}\sum_{m=0}^k \tilde{C}^i_m(t) \rightarrow C^i_{*}(t) \text{ for every point of continuity of } C^i_{*}(\cdot) \text{ and } t=T,\, \tilde{\Pb}_c \text{-a.s.}
\end{equation}
From now on, with a slight abuse of notation, we will denote by $C^i_{*}$ as well the right-continuous modification of $C^i_{*}$. Notice that having $\lim_{k\rightarrow \infty} I^i_k(t) = C^i_{*}(t)$ $\tilde{\Pb}_c$-a.s.\ for every point of continuity of $C^i_{*}(\cdot)$ and for $t=T$ (cf.\ \eqref{KomlosI}) means that the sequence of measures $dI^i_k(\omega,\cdot)$ on $[0,T]$ converges weakly to $dC^i_{*}(\omega,\cdot)$ $\tilde{\Pb}_c$-a.e. $\omega$; that is (see, e.g., \cite{Billingsley86})
\begin{equation}
\label{weakconvI}
\lim_{k \rightarrow \infty} \int_0^T  f(t)dI^i_k(t) = \int_0^T f(t)dC^i_{*}(t),\,\,\,\tilde{\Pb}_c-a.s.,
\end{equation}
for every bounded and $dC^i_{*}$-a.e.\ continuous function $f$.
We now claim that the Koml\'os' limit $(\underline{x}_{*}, \underline{C}_{*}):=(x^1_{*},\dots,x^n_{*},C^1_{*},\dots,C^n_{*})$ belongs to $\mathcal{B}_w$  and that it is optimal for the social planner's problem (\ref{SPproblem}).
Indeed, $(\underline{X}_{k},\underline{I}_{k}):=(X^1_{k},\dots,X^n_{k},I^1_{k},\dots,I^n_{k}) \in \mathcal{B}_w$ by convexity of $\mathcal{B}_w$. Moreover \eqref{convX}, \eqref{weakconvI} and Fatou's Lemma imply 
\begin{align}\label{BudgetLim}
w &\geq \liminf_{k \rightarrow \infty}\sum_{i=1}^n\Exp\biggl[\int_0^T \psi_x(t)X^i_k(t) dt + \int_0^T\psi_c(t) dI^i_k(t)\biggr] \nonumber\\
&= \liminf_{k \rightarrow \infty}\sum_{i=1}^n \biggl( \tilde{\Exp}_x \biggl[\int_0^T e^{-\int_0^t r_x(s)ds} X^i_k(t) dt\biggr] + \tilde{\Exp}_c \biggl[\int_0^T e^{-\int_0^t r_c(s)ds} dI^i_k(t)\biggr] \biggr) \nonumber\\
&\geq \sum_{i=1}^n \biggl(\tilde{\Exp}_x \biggl[\int_0^T e^{-\int_0^t r_x(s)ds} x^i_*(t) dt\biggr] + \tilde{\Exp}_c \biggl[\int_0^T e^{-\int_0^t r_c(s)ds} dC^i_*(t)\biggr]\biggr) \\
&= \sum_{i=1}^n \Exp\biggl[\int_0^T \psi_x(t)x^i_*(t) dt + \int_0^T\psi_c(t) dC^i_*(t)\biggr]; \nonumber
\end{align}
that is, $(\underline{x}_{*}, \underline{C}_{*}) \in \mathcal{B}_w$. 
Recall now that $\Pb_x \sim \Pb$ and $\Pb_c \sim \Pb$. Then (\ref{KomlosI}) and (\ref{convX}) also hold $\Pb$-a.s.\ and $d\Pb \otimes dt$-a.e., respectively (in particular $I^i_k(t)\rightarrow C^i_{*}(t)$ $dt$-a.e., $\Pb$-a.s.), and therefore we may write
$$\sum_{i=1}^n \gamma^i U^i(x^i_{*}, C^{i}_{*}; C^{-i}_{*}) = \lim_{k \rightarrow \infty} \sum_{i=1}^n \gamma^i U^i(X^i_k, I^i_k; I^{-i}_k) = V_{SP}$$
by the uniform integrability assumed in Assumption \ref{asm1}.\ref{unint}.\ and because $(\underline{X}_{k}, \underline{I}_{k})$ is also a maximizing sequence by concavity of each \(U^i\).
Hence $(\underline{x}_{*}, \underline{C}_{*})$ is optimal.\vspace{0.15cm}

\noindent The above arguments also extend to the infinite horizon case $T=+\infty$ as it is shown in the next step.\vspace{0.25cm}

\textsl{Step 2.}\quad We only sketch the proof, since some of the arguments are similar to those employed in \textsl{Step 1}. Let $T=+\infty$ and $\{(\underline{x}_m, \underline{C}_m)\}_{m \in \mathbb{N}}$ be an admissible maximizing sequence for $V_{SP}$. For each $K\in\mathbb{N}$ we can apply the construction of \textsl{Step 1} to the interval $[0,K]$ to obtain measures $\tilde\Pb^K_x$ and $\tilde\Pb^K_c$ equivalent to $\Pb$ on $(\Omega,\F_K)$ and limit processes $x^{K,i}_*$ and optional random measures $dC^{K,i}_*$ such that the convergence \eqref{convX} resp.\ \eqref{KomlosI} holds on $[0,K]$. As the convergence occurs also $d\Pb \otimes dt$ a.e., resp.\ \nbd{\Pb}a.s., we can aggregate the limits $\{x^{K,i}_*\}_{K\in \mathbb{N}}$ and $\{dC^{K,i}_*\}_{K\in \mathbb{N}}$ consistently to optional processes $x^{i}_*$ and optional measures $dC^{i}_*$ while maintaining convergence. Now we can apply the budget estimate \eqref{BudgetLim} on each $[0,K]$, let $K\to\infty$ and conclude by monotone convergence that $(\underline{x}_{*}, \underline{C}_{*}) \in \mathcal{B}_w$. Optimality obtains exactly as in \textsl{Step 1}. \vspace{0.25cm}

\textsl{Step 3.}\quad Finally, uniqueness of $(\underline{x}_{*}, C_{*})$ up to indistinguishability follows as usual (both in the finite and in the infinite time-horizon case) from strict concavity of the utility functions $u^i$, $i=1,\dots,n$, and from convexity of $\mathcal{B}_w$.
\end{proof}

We now aim to characterize the social planner's optimal policy by means of a set of first order conditions for optimality. These conditions may be thought of as a stochastic, infinite dimensional generalization of the classical Kuhn-Tucker method and they have been used in diverse instances to solve singular stochastic control problems of the monotone follower type (cf.\ the overview in Section \ref{sec:intro}).

For any $(\underline{x}, \underline{C}) \in \mathcal{B}_w$ and for some Lagrange multiplier $\lambda > 0$, define the \textsl{Lagrangian functional} of problem (\ref{SPproblem}) as
\begin{align*}
\mathcal{L}^w(\underline{x},\underline{C}; \lambda):= 
&\sum_{i=1}^n \gamma^i \Exp\left[\int_0^T e^{-\int_0^t r(s)\,ds}u^{i}(x^i(t), C(t))\,dt\right] \\
&+ \,\lambda\biggl\{ w - \Exp\biggl[\int_0^T \psi_x(t) x(t) dt + \int_0^T  \psi_c(t) dC(t)\biggr]\biggr\},
\end{align*}
where again $x:= \sum_{i=1}^n x^i$ and $C:=\sum_{i=1}^n C^i$. Moreover, let $\mathcal{T}$ be the set of all $(\F_t)$-stopping times with values in $[0,T]$ a.s.\ and denote by $\nabla_{c} \mathcal{L}^w$ the Lagrangian functional's supergradient with respect to the aggregated public good contribution; that is, the unique optional process such that
\begin{equation*}
\nabla_{c} \mathcal{L}^w (\underline{x},\underline{C}; \lambda)(\tau):= \Exp\biggl[\int_{\tau}^T e^{-\int_0^t r(s)\,ds} \sum_{i=1}^n \gamma^i\,u^{i}_c(x^i(t), C(t))\,dt\biggm|\F_{\tau}\biggr] - \lambda\psi_c(\tau)
\end{equation*}
for any $\tau \in \mathcal{T}$. $\nabla_{c} \mathcal{L}^w (\underline{x},\underline{C}; \lambda)(\tau)$ may be interpreted as the future marginal expected utility net of the marginal contribution cost, which the social planner would incur by an infinitesimal investment into the public good at time $\tau \in \mathcal{T}$ when the investment plan is $(\underline{x},\underline{C}) \in \mathcal{B}_w$ and the shadow value of the budget is $\lambda$.

On the other hand, additional consumption of the private good $x^i$ affects marginal utility only at those times at which consumption actually occurs, thus leading to 
\begin{equation*}
\nabla_{x} \mathcal{L}^w (\underline{x},\underline{C}; \lambda)(t) := \gamma^i e^{-\int_0^{t} r(s)ds } u^i_{x}(x^i(t), C(t)) - \lambda \psi_x(t), \quad t \in [0,T].
\end{equation*}

\begin{rem}
\label{rem:supergrad}
Following Remark 3.1 in \cite{BankRiedel01}, $\nabla_{c}\mathcal{L}^w(\underline{x},\underline{C}; \lambda)$ is the Riesz representation of the Lagrangian gradient at $C$.
More precisely, for any arbitrary but fixed $\lambda>0$, define $\nabla_{c}\mathcal{L}^w(\underline{x},\underline{C}; \lambda)$ as the optional projection of the product-measurable process
\begin{equation*}
\Phi(\omega,t):= \int_{t}^T e^{-\int_0^s r(\omega, u)\,du} \sum_{i=1}^n \gamma^i\,u^{i}_c(\omega, x^i(\omega,s), C(\omega,s))\,ds - \,\lambda\psi_c(\omega, t)
\end{equation*}
for $\omega \in \Omega$ and $t \in [0,T]$. Hence $\nabla_{c}\mathcal{L}^w(\underline{x},\underline{C}; \lambda)$ is uniquely determined up to $\Pb$-indistinguishability and it holds
$$\mathbb{E}\biggl\{\,\int_{0}^T \nabla_{c}\mathcal{L}^w(\underline{x},\underline{C}; \lambda)dC(t)\biggr\} = \mathbb{E}\biggl\{\,\int_{0}^T\Phi(t) dC(t)\biggr\}$$
for all admissible $C$ \citep[cf.][Theorem 1.33]{Jacod79}.
\end{rem}

\begin{prop}
\label{prop1SP}
Let Assumptions \ref{asm1} and \ref{asm2} hold. An admissible policy $(\underline{x}_{*}, \underline{C}_{*})$ is optimal for the social planner's problem (\ref{SPproblem}) if and only if there exists a Lagrange multiplier $\lambda_*>0$ such that the following first order conditions (FOCs) hold true 
\begin{equation}
\label{FOC}
\left\{
\begin{array}{ll}
\displaystyle \Exp\biggl[\int_0^T \psi_x(t) x_{*}(t) dt + \int_0^T  \psi_c(t) dC_{*}(t)\biggr] = w, \\ \\
\displaystyle \Exp\biggl[\int_{\tau}^T e^{-\int_0^t r(s)\,ds} \sum_{i=1}^n \gamma^i u^{i}_c(x^i_{*}(t), C_{*}(t))\,dt\biggm|\F_{\tau}\biggr] \leq \lambda_*\psi_c(\tau)\,\,\,\Pb\text{-a.s.}\;\forall \,\,\tau \in \mathcal{T},  \\ \\
\displaystyle \Exp\biggl[\int_0^T \biggl( \Exp\biggl[\int_{t}^T e^{-\int_0^s r(u)\,du} \sum_{i=1}^n \gamma^i u^{i}_c(x^i_{*}(s), C_{*}(s))\,ds\biggm|\F_{t}\biggr] - \lambda_*\psi_c(t)\biggr)dC_{*}(t)\biggr] = 0, \\ \\
\displaystyle \gamma^i e^{-\int_0^{t} r(s)ds } u^i_{x}(x^i_{*}(t), C_{*}(t)) \leq \lambda_* \psi_x(t)\,\,\,\Pb\otimes dt\text{-a.e.},\, \text{with equality on } \{x^i_{*}>0\}. 
\end{array}
\right.
\end{equation}
\end{prop}

The proof of Proposition \ref{prop1SP} is given in Appendix \ref{AppProofs}, Section \ref{proofFOCSprop}. It generalizes that of Theorem $3.2$ in \cite{BankRiedel01} to the present setting of a multidimensional optimal consumption problem with both classical and monotone controls and it is obtained by suitably adapting the proof of the classical Kuhn-Tucker conditions to the stochastic, infinite-dimensional setting. Indeed, concavity of the utility functions $u^i$, $i=1,\dots,n$, yields sufficiency, whereas the proof of the necessity part is more delicate. We linearize the original problem (\ref{SPproblem}) around its optimal solution $(\underline{x}_{*}, \underline{C}_{*})$ and then show that $(\underline{x}_{*}, \underline{C}_{*})$ solves the linearized problem as well. It is then proved that any solution to the linearized problem (and therefore $(\underline{x}_{*}, \underline{C}_{*})$ as well) satisfies some flat-off conditions as the third and the fourth of (\ref{FOC}).

Notice that because of the Inada condition $\lim_{x\downarrow 0}u^i_x(x,c)=\infty$ (cf.\ Assumption \ref{asm1}.\ref{Ui_x}.), the fourth one of \eqref{FOC} must be always binding, i.e.
\begin{equation*}
\gamma^i e^{-\int_0^{t} r(s)ds } u^i_{x}(x^i_{*}(t), C_{*}(t)) = \lambda_* \psi_x(t) \quad \Pb\otimes dt\text{-a.e.}
\end{equation*}
Thus, recalling that $(x,c) \mapsto u^i(x,c)$ is strictly concave (see again Assumption \ref{asm1}) and denoting by $g^i(\cdot, c)$ the inverse of $u^i_x(\cdot, c)$, we may write
\beq
\label{binding2}
x^i_{*}(t)=g^i\Big(\frac{\lambda_*}{\gamma^i} e^{\int_0^{t} r(s)ds } \psi_x(t), C_{*}(t)\Big) \quad \Pb \otimes dt\text{-a.e.}
\eeq
Then, by plugging (\ref{binding2}) into (\ref{FOC}) we obtain the equivalent formulation of the first order conditions for optimality
\begin{equation}
\label{FOC2}
\left\{
\begin{array}{ll}
\displaystyle \Exp\biggl[\int_0^T \psi_x(t) \sum_{i=1}^n g^i\Big(\frac{\lambda_*}{\gamma^i} e^{\int_0^{t} r(s)ds } \psi_x(t), C_{*}(t)\Big)\, dt + \int_0^T  \psi_c(t) dC_{*}(t)\biggr] = w, \\ \\
\displaystyle \Exp\biggl[\int_{\tau}^T e^{-\int_0^t r(s)\,ds} \sum_{i=1}^n \gamma^i h^{i}\Big(\frac{\lambda_*}{\gamma^i} e^{\int_0^{t} r(u)du }\psi_x(t), C_{*}(t)\Big)\,dt \biggm|\F_{\tau}\biggr] \leq \lambda_*\psi_c(\tau)\,\,\,\Pb\text{-a.s.}\;\forall \,\,\tau \in \mathcal{T},  \\ \\
\displaystyle \Exp\biggl[\int_0^T  \biggl( \Exp\biggl[\int_{t}^T e^{-\int_0^s r(u)\,du} \sum_{i=1}^n \gamma^i h^i\Big(\frac{\lambda_*}{\gamma^i} e^{\int_0^{s} r(u)du }\psi_x(s), C_{*}(s)\Big)\,ds\biggm|\F_{t}\biggr] - \lambda_*\psi_c(t)\biggr)dC_{*}(t)\biggr] = 0,
\end{array}
\right.
\end{equation}
with $h^i(\psi,c):=u^i_c(g^i(\psi,c),c)$. In Lemma \ref{fn:h_c<0} in the appendix some properties of $h^i$ are proved.

Although the first order conditions of Proposition \ref{prop1SP} (or those in (\ref{FOC2})) completely characterize the optimal policy, they are not binding at all times and so they cannot be directly used to determine $C_{*}$ and consequently $\underline{x}_{*}$ by (\ref{binding2}). As usual in the literature on stochastic control problems with monotone controls (see, e.g., \cite{ChiarollaHaussmann94}, \cite{ElkarouiKaratzas91}, \cite{Karatzas81}, \cite{KaratzasShreve84} as classical references), the optimal policy consists of keeping the controlled process close to some barrier (which is the free boundary of an associated optimal stopping problem in a Markovian setting) in a `minimal way' (i.e.\ according to a Skorokhod's reflection principle). Here we derive the social planner's optimal investment into the public good $C_{*}$ in terms of the running supremum of an index process representing the desirable value of investment or consumption the agents would like to have. Mathematically, such an index process is the optional solution of a stochastic backward equation in the spirit of Bank-El Karoui (cf.\ \cite{BankElKaroui04}, Theorem $1$ and Theorem $3$) and it may be represented in terms of the value functions of a family of optimal stopping problems (see also \cite{BankFoellmer03} for further details and applications).

The following result is proved in Appendix \ref{AppProofs}, Section \ref{proof:prop:existenceback}.

\begin{prop}
\label{prop:existenceback}
Define for every \(i=1,\dots,n\) \(g^i(\cdot,c)\) as the inverse of \(u_x^i(\cdot,c)\), as well as \(h^i(\psi,c):=u_c^i(g(\psi,c),c)\) for any \(\psi,c>0\). Let Assumptions \ref{asm1} and \ref{asm2} hold and suppose that the Inada conditions $\lim_{c \downarrow 0}h^i(\psi,c)=+\infty$ and $\lim_{c \uparrow \infty}h^i(\psi,c)=0$ are satisfied. Then for any given set of weights $\{\gamma_i\}_{i=1,...,n}$ and for any given Lagrange multiplier $\lambda>0$ there exists a unique (up to indistinguishability) optional, upper right-continuous process $l^{*}$ solving the stochastic backward equation
\begin{equation}
\label{SPback}
\Exp\biggl[\int_{\tau}^T e^{-\int_0^t r(s)\,ds} \sum_{i=1}^n \gamma^i h^{i}\biggl(\frac{\lambda}{\gamma^i} e^{\int_0^{t} r(u)du }\psi_x(t), \sup_{\tau \leq u \leq t} l(u)\biggr)\,dt \biggm|\F_{\tau}\biggr] = \lambda\psi_c(\tau)\mathds{1}_{\{\tau < T\}} 
\end{equation}
for any stopping time $\tau \in [0,T]$ subject to $l^{*}_T=0$, \(\Pb\)-a.s.
\end{prop}

Notice that differently to those on $u^i$ (cf.\ Assumption \ref{asm1}), Inada conditions on $h^i$ do not have an immediate economic interpretation. However these requirements have a clear mathematical importance, as they suffice to show existence of a unique optional, upper right-continuous $l^{*}$ solving \eqref{SPback} (see Section \ref{proof:prop:existenceback} of Appendix \ref{AppProofs} for details).  
\begin{thm}
\label{SPoptpolicy}
Under the assumptions of Proposition \ref{prop:existenceback}, $(x^1_*,\dots,x^n_*,C^1_*,\dots,C^n_*)$ solves the social planner's problem \eqref{SPproblem} if and only if 
\begin{equation}
\label{optimalSP}
\left\{
\begin{array}{ll}
C_{*}(t)=\sum_{i=1}^n C^i_{*}(t)=\sup_{0 \leq u \leq t}l^{*}(u) \vee 0, \\ \\
x^i_{*}(t)=g^i\Big(\frac{\lambda_*}{\gamma^i} \psi_x(t), C_{*}(t)\Big),\,\,\,\,i=1,\dots,n,
\end{array}
\right.
\end{equation}
where $l^*$ is the unique solution of \eqref{SPback} for a suitable Lagrange multiplier $\lambda_*>0$ such that $\sum_{i=1}^n\Psi(x^i_{*},C^i_{*})=w$.
\end{thm}

\begin{proof}
To show optimality of $(\underline{x}_{*}(t),C_{*}(t))$ as in (\ref{optimalSP}) it suffices to verify that it is admissible and satisfies the (necessary and) sufficient first order conditions (\ref{FOC2}).
By Theorem 33 in Chapter IV of \cite{DellacherieMeyer78}, $C_{*}$ as in (\ref{optimalSP}) is progressively measurable since $l^{*}$ is optional; also, it has right-continuous sample paths since $l^{*}$ is upper right-continuous. Hence it is optional. On the other hand, $x^i_{*}$, $i=1,\dots,n$, is optional and positive since $g^i$ is continuous and positive. Moreover, for any $\tau \in [0,T)$ we have
\begin{align}\label{checkFOCs}
&\Exp\biggl[\int_{\tau}^T e^{-\int_0^t r(s)\,ds} \sum_{i=1}^n \gamma^i h^{i}\Big(\frac{\lambda_*}{\gamma^i} e^{\int_0^{t} r(u)du }\psi_x(t), \sup_{0 \leq u \leq t}l^{*}(u) \vee 0\Big)\,dt \biggm|\F_{\tau}\biggr] \notag\\
\leq\; &\Exp\biggl[\int_{\tau}^T e^{-\int_0^t r(s)\,ds} \sum_{i=1}^n \gamma^i h^{i}\Big(\frac{\lambda_*}{\gamma^i} e^{\int_0^{t} r(u)du }\psi_x(t), \sup_{\tau \leq u \leq t}l^{*}(u)\Big)\,dt \biggm|\F_{\tau}\biggr] = \lambda_* \psi_c(\tau),
\end{align}
where the first inequality follows from the fact that $c \mapsto h^i(\psi,c)$ is strictly decreasing (see Lemma \ref{fn:h_c<0}), whereas (\ref{SPback}) implies the last equality. On the other hand, if $\tau \in \mathcal{T}$ is a time of investment, i.e.\ such that $C_{*}(\tau + \varepsilon) - C_{*}(\tau-) > 0$\footnote{That is, $\tau$ is a time of increase for $C_*(\omega,\cdot)$.} for any $\varepsilon>0$, we have $\sup_{0 \leq u \leq t}l^{*}(u) \vee 0 = \sup_{\tau \leq u \leq t}l^{*}(u)$ on $(\tau,T]$ and equality holds in (\ref{checkFOCs}). Since those times carry the measure $dC_*$, the third line of (\ref{FOC2}) is satisfied as well (note that $dC_*$ does not charge $[T]$ by $l^*_T=0$).

We now argue that the required $\lambda_*$ to satisfy the budget constraint $\sum_{i=1}^n\Psi(x^i_{*},C^i_{*})=w$ exists. Indeed, by Theorem \ref{SPexistence} we know a priori that a solution $(\underline{x}_{*}(t),C_{*}(t))$ to our control problem exists, and by necessity in Proposition \ref{prop1SP} there is a corresponding $\lambda_*$. That solution satisfies also the FOCs \eqref{FOC2}, of which we now ignore the first line. Then we have FOCs for a control problem \emph{without} budget constraint, but with linear investment cost $\lambda_*\psi_c\,dC$ and profit stream
\[
H(\omega,t,c):=e^{-\int_0^t r(\omega,s)\,ds} \sum_{i=1}^n \gamma^i \int_0^c h^{i}\Big(\frac{\lambda_*}{\gamma^i} e^{\int_0^{t} r(\omega,u)du }\psi_x(\omega,t), y\Big)\,dy,
\]
i.e., the problem of maximizing
\[
\Exp\biggl[\int_0^T H(t,C_t)\,dt - \int_0^T \lambda_*\psi_c(t)\,dC_t\biggr]
\]
over $C$ adapted, right-continuous and nondecreasing with $C(0-)=0$, \nbd{\Pb}a.s. This latter problem has a unique solution (by strict concavity) given by $(\underline{x}_{*}(t),C_{*}(t))$ since that solves the sufficient FOCs in \eqref{FOC2}. On the other hand, our solution from \eqref{optimalSP} (now with $\lambda_*$ from Proposition \ref{prop1SP}) also solves those FOCs, as verified above. Hence, the latter solution must coincide with $(\underline{x}_{*}(t),C_{*}(t))$ from Theorem \ref{SPexistence} and thus be budget-feasible.

Finally, necessity follows from the uniqueness established in Theorem \ref{SPexistence}.
\end{proof}

\begin{rem}
\label{onthesignalprocess}
The process $l^{*}$ may be found numerically by backward induction on a discretized version of problem \eqref{SPback} (see \cite{BankFoellmer03}, Section $4$). In some cases, when $T=+\infty$, \eqref{SPback} has a closed form solution as in the case of a Cobb-Douglas utility function (see Section \ref{freerider} below).
\end{rem}


\section{The Public Good Contribution Game}
\label{game}

In Section \ref{SocialPlanner} we have taken the point of view of a fictitious social planner aiming to efficiently maximize the social welfare. Here we aim to study strategic interaction between the agents of our economy. Therefore assume from now on $n>1$.

Determining agent \(i\)'s optimal choice of a strategy against a given process \(C^{-i}\) specifying aggregate contributions by the opponents amounts to solving the stochastic control problem with value function
\begin{equation}
\label{Vi}
V^{i}(C^{-i}):=\sup_{(x^i,C^{i})\in\mathcal{B}_{w^i}}U^{i}(x^i,C^{i};C^{-i}), \qquad i=1,\dots,n,
\end{equation}
where $\mathcal{B}_{w^i}$ and $U^i$ are as in (\ref{Bwi}) and (\ref{payoff}), respectively.
The description of the game is completed by the introduction of a standard Nash equilibrium concept.
\begin{defn}
\label{precNash}
\((\hat{x}^{1},\dots,\hat{x}^{n},\hat{C}^{1},\dots,\hat{C}^{n})\) is a \emph{Nash equilibrium} if for all \(i\in\{1,\dots,n\}\), \((\hat{x}^{i},\hat{C}^{i})\in\mathcal{B}_{w^i}\) and \(U^{i}(\hat{x}^{i}, \hat{C}^{i},\hat{C}^{-i})=V^{i}(\hat{C}^{-i})\).
\end{defn}

While this equilibrium notion does not limit the ability of any agent to optimize against given strategies of the others, it does limit the extent of dynamic interaction that can take place. Although agents do react to the evolving exogenous uncertainty, they take the contribution processes of others as given and do not react to deviations from announced (equilibrium) play. Therefore, one might term such an equilibrium as one in \emph{precommitment strategies}. Unfortunately there are serious conceptual difficulties in defining a related game with more explicit feedback strategies as argued by \cite{BackPaulsen09}, which is why we consider simple Nash equilibria here. Besides these conceptual problems, the choice of open-loop strategies we provide here can be justified at the modeling stage if agents are not able to observe the opponents' investments in the public good.

As in the social planner's case we shall first characterize solutions of the best reply problems (\ref{Vi}) by means of a stochastic Kuhn-Tucker approach. The next Proposition accomplishes this. Its proof may be obtained by adopting arguments similar to those employed to prove Proposition \ref{prop1SP} and therefore we omit its proof for the sake of brevity.

\begin{prop}\label{prop1game}
Let $\hat{C}^{-i}$ be given and Assumptions \ref{asm1} and \ref{asm2} hold. Then $(\hat{x}^i, \hat{C}^i) \in \mathcal{B}_{w^i}$ attains $V^{i}(\hat{C}^{-i})$ (cf.\ (\ref{Vi})) if and only if there exists a Lagrange multiplier $\hat{\lambda}^i>0$ such that the following first order conditions hold true
\begin{equation}
\label{FOCgame}
\left\{
\begin{array}{ll}
\displaystyle \Exp\biggl[\int_0^T \psi_x(t) \hat{x}^i(t) dt + \int_0^T  \psi_c(t) d\hat{C}^i(t)\biggr] = w^i, \\ \\
\displaystyle \Exp\biggl[\int_{\tau}^T e^{-\int_0^t r(s)\,ds} u^{i}_c(\hat{x}^i(t), \hat{C}(t))\,dt\biggm|\F_{\tau}\biggr] \leq \hat{\lambda}^i\psi_c(\tau)\,\,\,\Pb\text{-a.s.}\;\forall \,\,\tau \in \mathcal{T},  \\ \\
\displaystyle \Exp\biggl[\int_0^T  \biggl( \Exp\biggl[\int_{t}^T e^{-\int_0^s r(u)\,du}  u^{i}_c(\hat{x}^i(s), \hat{C}(s))\,ds\biggm|\F_{t}\biggr] - \hat{\lambda}^i\psi_c(t)\biggr)d\hat{C}^i(t)\biggr] = 0,\\ \\
\displaystyle e^{-\int_0^{t} r(u)\,du} u^i_{x}(\hat{x}^i(t), \hat{C}(t)) \leq \hat{\lambda}^i \psi_x(t)\,\,\,\Pb\otimes dt\text{-a.e.},\,\text{with equality on } \{\hat{x}^i>0\}. 
\end{array}
\right.
\end{equation}
\end{prop}

Again, the Inada conditions (cf.\ Assumption \ref{asm1}.\ref{Ui_x}.) imply that the fourth one of (\ref{FOCgame}) is always binding. Hence, we may equivalently rewrite (\ref{FOCgame}) as
\begin{equation}
\label{FOC2game}
\left\{
\begin{array}{ll}
\displaystyle \Exp\biggl[\int_0^T \psi_x(t) g^{i}(\hat{\lambda}^i e^{\int_0^{t} r(s)\,ds}\psi_x(t), \hat{C}(t))\, dt + \int_0^T  \psi_c(t) d\hat{C}^i(t)\biggr] = w^i, \\ \\
\displaystyle \Exp\biggl[\int_{\tau}^T e^{-\int_0^t r(s)\,ds}  h^{i}(\hat{\lambda}^i e^{\int_0^{t} r(s)\,ds}\psi_x(t), \hat{C}(t))\,dt \biggm|\F_{\tau}\biggr] \leq \hat{\lambda}^i \psi_c(\tau)\,\,\,\Pb\text{-a.s.}\;\forall \,\,\tau \in \mathcal{T},  \\ \\
\displaystyle \Exp\biggl[\int_0^T  \biggl( \Exp\biggl[\int_{t}^T e^{-\int_0^s r(u)\,du}  h^i(\hat{\lambda}^i e^{\int_0^{s} r(u)\,du}\psi_x(s), \hat{C}(s))\,ds\biggm|\F_{t}\biggr] - \hat{\lambda}^i\psi_c(t)\biggr)d\hat{C}^i(t)\biggr] = 0,
\end{array}
\right.
\end{equation}
where again $h^i(\psi,c):=u^i_c(g^i(\psi,c),c)$ with $g^i(\cdot,c)$ the inverse of $u^i_x(\cdot,c)$.

A more explicit characterization of a Nash equilibrium (if it exists) in the spirit of Theorem \ref{SPoptpolicy} for the social planner's optimal policy seems difficult to obtain in the generality of our public good contribution game. In fact the first order conditions in \eqref{FOC2game} are mutually interconnected and this leads to a daunting (infinite-dimensional) fixed point problem.

However, when all the agents have the same utility function and the same initial wealth, we can make quite clear predictions by our approach. We can now establish existence of a symmetric Nash equilibrium in which $(\hat x^i,\hat C^i)=(\hat x^1,\hat C^1)$ for all $i=1,\dots,n$. We can relate all other equilibria to the symmetric one, giving us also a qualitative uniqueness. Therefore we exploit the fact that in the symmetric case the first order conditions of Proposition \ref{prop1game} are the same for all agents (including $\hat\lambda^i$) and they have a similar form as those for the social planner with equal weights (up to some factor $n$). Nevertheless, we cannot transform equilibrium determination into a control problem. For these results we make the additional assumption that $u_{xc}$ is nonnegative, which holds for instance in the standard setting of quasilinear (i.e., additively separable) utility in the economics literature or also for Cobb-Douglas utilities considered below.

\begin{thm}
\label{thm:eqlunique}
Let Assumption \ref{asm2} hold and suppose that all agents have the same initial wealth $w^i=w/n>0$ and the same utility function $u^i\equiv u$ satisfying Assumption \ref{asm1}. Further assume $u_{xc}\geq 0$\footnote{Mathematically, the utility function $u$ is supermodular (see \cite{Topkis}). Economically, this means that the private and the public goods are complements. Two goods are complements if the increase of a good's demand corresponds to a decrease in the price of the other good.}. Then there exists a symmetric Nash equilibrium \((\hat{x}^{1},\dots,\hat{x}^{n},\hat{C}^{1},\dots,\hat{C}^{n})\) with $(\hat x^i,\hat C^i)=(\hat x^1,\hat C^1)$ for all $i=1,\dots,n$. Furthermore, any tuple of processes \((\tilde{x}^{1},\dots,\tilde{x}^{n},\tilde{C}^{1},\dots,\tilde{C}^{n})\) is a Nash equilibrium of the public good contribution game \eqref{Vi} if and only if $(\tilde x^i,\tilde C^i)\in\mathcal{B}_{w^i}$ with $\tilde x^i=\hat x^i$ $P\otimes dt$-a.e.\ and $\Exp\bigl[\int_0^T \psi_c(t) d\tilde C^i(t)\bigr]=\Exp\bigl[\int_0^T \psi_c(t) d\hat C^i(t)\bigr]$ for all $i=1,\dots,n$, and $\sum_i\tilde C^i(t)=\sum_i\hat C^i(t)$ for all $t\in[0,T)$, $\Pb$-a.s.
\end{thm}

\begin{proof}
We first prove the second claim in five steps. Afterwards we will prove existence of a symmetric Nash equilibrium in two further steps.\vspace{0.25cm}

\emph{Step 1.}\,\,\,As preliminary, notice that in \emph{any} equilibrium \((\hat{x}^{1},\dots,\hat{x}^{n},\hat{C}^{1},\dots,\hat{C}^{n})\), every agent $i=1,\dots,n$ must exhaust his budget, $w^i=\Exp\bigl[\int_0^T \psi_x(t) \hat{x}^i(t) dt +\int_0^T\psi_c(t)\,d\hat C^i(t) \bigr]$, by the first of \eqref{FOCgame}, but not incur any costs for a final contribution, $E[\psi_c(T)\,\Delta\hat C^i(T)]=0$ (where $\Delta\hat C^i(T):=\hat C^i(T)-\hat C^i(T-)$). Indeed, since $\Delta\hat C^i(T)$ does not affect $U^{i}(\hat x^i,\hat C^i;\hat C^{-i})$ in \eqref{payoff} and $u^i$ is strictly increasing in $x$ under Assumption \ref{asm1}, if $\delta:=E[\psi_c(T)\,\Delta\hat C^i(T)]>0$, then $(\tilde x^i,\tilde C^i)=(\hat x^i+\delta/\Exp[\int_0^T \psi_x(t) dt],\hat C^i-\Delta\hat C^i(T)\indi{t=T})\in\mathcal{B}_{w^i}$ would yield $U^{i}(\tilde x^i,\tilde C^i;\hat C^{-i})>U^{i}(\hat x^i,\hat C^i;\hat C^{-i})$.
\vspace{0.25cm}

\emph{Step 2.}\,\,\,Now we show that under the assumption $u_{xc}\geq 0$, we must have symmetric Lagrange multipliers $\hat\lambda^1=\cdots=\hat\lambda^n$ in the (necessary) first order conditions from Proposition \ref{prop1game} in \emph{any} equilibrium \((\hat{x}^{1},\dots,\hat{x}^{n},\hat{C}^{1},\dots,\hat{C}^{n})\), and also symmetric private consumption \(\hat{x}^{1}=\cdots=\hat{x}^{n}\) $P\otimes dt$-a.e. 

In any equilibrium, every agent $i=1,\dots,n$ must have a Lagrange multiplier $\hat\lambda^i$ such that the first order conditions \eqref{FOC2game} hold. By the third condition for an arbitrary agent $i$ and the second condition for any other agent $j$,
\begin{align*}
&\Exp\biggl[\int_0^T \biggl( \Exp\biggl[\int_{t}^T e^{-\int_0^s r(u)\,du}\frac{1}{\hat\lambda^i}  h(\hat{\lambda}^i e^{\int_0^{s} r(u)\,du}\psi_x(s), \hat{C}(s))\,ds\biggm|\F_{t}\biggr] - \psi_c(t)\biggr)d\hat{C}^i(t)\biggr] = 0 \\
\geq{}&\Exp\biggl[\int_0^T \biggl( \Exp\biggl[\int_{t}^T e^{-\int_0^s r(u)\,du}\frac{1}{\hat\lambda^j}  h(\hat{\lambda}^j e^{\int_0^{s} r(u)\,du}\psi_x(s), \hat{C}(s))\,ds\biggm|\F_{t}\biggr] - \psi_c(t)\biggr)d\hat{C}^i(t)\biggr].
\end{align*}
This inequality can only hold if $E[\hat C^i_{T-}]=0$ (then indeed equality holds) or if $\hat \lambda^j\geq \hat \lambda^i$, because the mapping $\lambda \mapsto \frac{1}{\lambda}h(\lambda\psi,c)$ is strictly decreasing under the assumption $u_{xc}\geq 0$ by Lemma \ref{fn:h_c<0} in the appendix. Thus, the $\hat\lambda^i$ of all agents $i$ with $E[\hat C^i_{T-}]>0$ must be minimal and hence identical. Then also the private consumptions $\hat x^i(t)=g\bigl(\hat \lambda^i \psi_x(t), \hat C(t)\bigr)$ $P\otimes dt$-a.e.\ must be identical among these agents. 

Now consider an agent $j$ with $E[\hat C^j_{T-}]=0$ and an arbitrary other agent $i$. Suppose $\hat\lambda^j \geq \hat\lambda^i$ (which must hold if $E[\hat C^i_{T-}]>0$ as we have already shown, or which is without loss of generality if also $E[\hat C^i_{T-}]=0$). Then $\hat x^j(t)=g(\hat\lambda^j e^{\int_0^{t} r(s)\,ds}\psi_x(t), \hat{C}(t)) \leq g(\hat\lambda^i e^{\int_0^{t} r(s)\,ds}\psi_x(t), \hat{C}(t))=\hat x^i(t)$ $\Pb\otimes dt$-a.e.\ by the monotonicity of $g(\psi,c)$ in $\psi$; the inequality is even strict if $\hat\lambda^j > \hat\lambda^i$. 

By \emph{Step 1}, $j$ exhausts his budget on private consumption, $w^j=w/n=\Exp\bigl[\int_0^T \psi_x(t) \hat{x}^j(t) dt \bigr] \leq \Exp\bigl[\int_0^T \psi_x(t) \hat{x}^i(t) dt +\int_0^T\psi_c(t)\,d\hat C^i(t) \bigr]=w^i=w/n$. As the inequality cannot be strict, we must in fact have $\hat\lambda^j=\hat\lambda^i$, and thus also $\hat x^i(t)=g\bigl(\hat \lambda^i \psi_x(t), \hat C(t)\bigr)$ $P\otimes dt$-a.e.\ must be the same for all agents.
\vspace{0.25cm}

\emph{Step 3.}\,\,\,Consider again \emph{any} Nash equilibrium \((\hat{x}^{1},\dots,\hat{x}^{n},\hat{C}^{1},\dots,\hat{C}^{n})\), whence the necessary and sufficient first order conditions \eqref{FOCgame} hold for all agents $i=1,\dots,n$ with symmetric Lagrange multipliers $\hat\lambda^1=\cdots=\hat\lambda^n$ and also $\hat x^1=\cdots=\hat x^n$ $P\otimes dt$-a.e.\ by \emph{Step 2}. We now show that if another tuple \((\tilde{x}^{1},\dots,\tilde{x}^{n},\tilde{C}^{1},\dots,\tilde{C}^{n})\) satisfies $(\tilde x^i,\tilde C^i)\in\mathcal{B}_{w^i}$ with $\tilde x^i=\hat x^i$ $P\otimes dt$-a.e.\ and $\Exp\bigl[\int_0^T \psi_c(t) d\tilde C^i(t)\bigr]=\Exp\bigl[\int_0^T \psi_c(t) d\hat C^i(t)\bigr]$ for all $i=1,\dots,n$, and $\sum_i\tilde C^i(t)=\sum_i\hat C^i(t)$ for all $t\in[0,T)$, $\Pb$-a.s., then the first order conditions \eqref{FOCgame} also hold for \((\tilde{x}^{1},\dots,\tilde{x}^{n},\tilde{C}^{1},\dots,\tilde{C}^{n})\) with $\tilde\lambda^i=\hat\lambda^1$, $i=1,\dots,n$. This will prove sufficiency in the second claim of the theorem.

Indeed, the first condition of \eqref{FOCgame} holds by the feasibility hypothesis $(\tilde x^i,\tilde C^i)\in\mathcal{B}_{w^i}$. The second and fourth of \eqref{FOCgame} are identical to those from the initial equilibrium as $\tilde x^i=\hat x^i$ and $\tilde C=\hat C$ $P\otimes dt$-a.e., and $\tilde\lambda^i=\hat\lambda^i$ by hypothesis. If one takes the sum of the third condition of \eqref{FOCgame} over all $i=1,\dots,n$ in the initial equilibrium, where now only $d\hat C^i$ can depend on $i$, then one gets
\beq
\label{Foc:eqlbinding}
\Exp\biggl[\int_0^T  \biggl( \Exp\biggl[\int_{t}^T e^{-\int_0^s r(u)\,du}  u_c(\hat{x}^1(s), \hat{C}(s))\,ds\biggm|\F_{t}\biggr] - \hat{\lambda}^1\psi_c(t)\biggr)d\hat{C}(t)\biggr] = 0.
\eeq
As the measure $d\tilde C^i$ is absolutely continuous with respect to $d\tilde C=d\hat C$ on $[0,T)$, \eqref{Foc:eqlbinding} implies the third condition of \eqref{FOCgame} for $\tilde x^i=\hat x^1$, $\tilde C=\hat C$ ($P\otimes dt$-a.e.), $\tilde\lambda^i=\hat\lambda^1$ and $d\tilde C^i$, up to a possible terminal mass $\Delta\tilde C^i(T)$. By \emph{Step 1}, however, $E[\psi_c(T)\,\Delta\hat C^i(T)]=0$, thus $\Exp\bigl[\int_0^T \psi_c(t) d\hat C(t)\bigr]=\Exp\bigl[\int_{[0,T)} \psi_c(t) d\hat C(t)\bigr]=\Exp\bigl[\int_{[0,T)} \psi_c(t) d\tilde C(t)\bigr]$ in sum. By hypothesis, also $\Exp\bigl[\int_0^T \psi_c(t) d\hat C(t)\bigr]=\Exp\bigl[\int_0^T \psi_c(t) d\tilde C(t)\bigr]$, so $E[\psi_c(T)\,\Delta\tilde C(T)]=0$ and hence $E[\psi_c(T)\,\Delta\tilde C^i(T)]=0$, which completes the third condition of \eqref{FOCgame} for $\tilde x^i$, $\tilde C$, $\tilde\lambda^i$ and $d\tilde C^i$.

Notice that one can set in particular $\tilde C^i=\hat C/n$, $i=1,\dots,n$. Indeed, $\Exp\bigl[\int_0^T \psi_c(t) d\hat C^i(t)\bigr]=w/n-\Exp\bigl[\int_0^T \psi_x(t) \hat{x}^i(t) dt\bigr]$ are identical for all agents in equilibrium, as the $\hat x^i$ are identical by \emph{Step 2} and the budget constraints binding by \emph{Step 1}; thus, summing up, $\Exp\bigl[\int_0^T\psi_c(t)\,d\hat C(t) \bigr]=w-n\Exp\bigl[\int_0^T \psi_x(t) \hat{x}^1(t) dt\bigr]=n\Exp\bigl[\int_0^T\psi_c(t)\,d\hat C^1(t) \bigr]$ and $(\hat x^i,\hat C/n)\in\mathcal{B}_{w/n}$.
\vspace{0.25cm}

\emph{Step 4.}\,\,\,Now we can establish uniqueness of $\hat C$ on $[0,T)$ in equilibrium. Consider two equilibria with aggregate $\hat C$ and $\tilde C$ and symmetric Lagrange multipliers $\hat\lambda^1=\cdots=\hat\lambda^n=:\hat\lambda$ and $\tilde\lambda^1=\cdots=\tilde\lambda^n=:\tilde\lambda$ by \emph{Step 2}. Let $\hat\lambda\geq\tilde\lambda$ without loss of generality. Only employing the last two of \eqref{FOC2game} we now show that we then have a uniform ordering $\hat C\leq\tilde C$, $P\otimes dt$-a.e. Therefore let $\sigma_{\varepsilon}:=\inf\{t\geq 0 \mid \hat C_t>\tilde C_t+\varepsilon\}\wedge T\in\mathcal{T}$ for arbitrary $\varepsilon>0$, and $\tau_{\varepsilon}:=\inf\{t\geq \sigma_{\varepsilon} \mid \tilde C_t>\hat C_t-\varepsilon/2\}\wedge T\in\mathcal{T}$. Then $\tau_{\varepsilon}>\sigma_{\varepsilon}$ on $\{\sigma_{\varepsilon}<T\}$ by right-continuity of $\hat C-\tilde C$ and $\hat C(t)>\tilde C(t)$ on $[\sigma_\varepsilon,\tau_\varepsilon)$ (all $\Pb$-a.s.). $\sigma_\varepsilon$ and $\tau_\varepsilon$ are points of increase of the processes $\hat C$ and $\tilde C$, respectively, and hence of $\hat C^i$ and $\tilde C^j$ for some players $i,j\in\{1,\dots,n\}$.\footnote{%
$\tau\in\mathcal{T}$ is a \emph{point of increase} of the optional, non-decreasing process $\hat C$ if $\hat C(t)>\hat C(\tau-)$ for all $t\in(\tau,T]$, $\Pb$-a.s.
}
Therefore, by the third condition of \eqref{FOC2game}, the second condition of \eqref{FOC2game} for $\hat C$, $\hat\lambda$ must be binding at $\tau=\sigma_\varepsilon$ on $\{\sigma_\varepsilon<T\}$ (see Proposition 3.2 of \cite{Steg12} for a proof). Then, on $\{\sigma_\varepsilon<T\}$ we obtain the following chain (up to a $\Pb$-nullset):
\begin{eqnarray*}
&&\Exp\biggl[\int_{\sigma_{\varepsilon}}^T e^{-\int_0^t r(s)\,ds} \frac{1}{\hat\lambda} h(\hat\lambda e^{\int_0^{t} r(s)\,ds}\psi_x(t), \hat{C}(t))\,dt \biggm|\F_{\sigma_{\varepsilon}}\biggr]  =  \psi_c(\sigma_{\varepsilon}) \nonumber \\
&&\geq \Exp\biggl[\int_{\sigma_{\varepsilon}}^T e^{-\int_0^t r(s)\,ds} \frac{1}{\tilde\lambda} h(\tilde\lambda e^{\int_0^{t} r(s)\,ds}\psi_x(t), \tilde{C}(t))\,dt \biggm|\F_{\sigma_{\varepsilon}}\biggr] \nonumber \\
&&>\Exp\biggl[\int_{\sigma_{\varepsilon}}^{\tau_{\varepsilon}} e^{-\int_0^t r(s)\,ds} \frac{1}{\tilde\lambda} h(\tilde\lambda e^{\int_0^{t} r(s)\,ds}\psi_x(t), \hat{C}(t))\,dt\biggm|\F_{\sigma_{\varepsilon}}\biggr]\nonumber 
\end{eqnarray*}
\begin{eqnarray}
\label{long}
&&\hspace{0.5cm}+\Exp\biggl[\int_{\tau_{\varepsilon}}^T e^{-\int_0^t r(s)\,ds} \frac{1}{\tilde\lambda} h(\tilde\lambda e^{\int_0^{t} r(s)\,ds}\psi_x(t), \tilde{C}(t))\,dt \biggm|\F_{\sigma_{\varepsilon}}\biggr] \nonumber \\
&&=\Exp\biggl[\int_{\sigma_{\varepsilon}}^{\tau_{\varepsilon}} e^{-\int_0^t r(s)\,ds} \frac{1}{\tilde\lambda} h(\tilde\lambda e^{\int_0^{t} r(s)\,ds}\psi_x(t), \hat{C}(t))\,dt \biggm|\F_{\sigma_{\varepsilon}}\biggr]+\Exp\bigl[\psi_c(\tau_{\varepsilon}) \bigm|\F_{\sigma_{\varepsilon}}\bigr] \nonumber \\
&&\geq\Exp\biggl[\int_{\sigma_{\varepsilon}}^{\tau_{\varepsilon}} e^{-\int_0^t r(s)\,ds} \frac{1}{\tilde\lambda} h(\tilde\lambda e^{\int_0^{t} r(s)\,ds}\psi_x(t), \hat{C}(t))\,dt \biggm|\F_{\sigma_{\varepsilon}}\biggr]  \\
&&\hspace{0.5cm}+\Exp\biggl[\int_{\tau_{\varepsilon}}^T e^{-\int_0^t r(s)\,ds} \frac{1}{\hat\lambda} h(\hat\lambda e^{\int_0^{t} r(s)\,ds}\psi_x(t), \hat{C}(t))\,dt \biggm|\F_{\sigma_{\varepsilon}}\biggr] \nonumber \\
&&\geq\Exp\biggl[\int_{\sigma_{\varepsilon}}^{T} e^{-\int_0^t r(s)\,ds} \frac{1}{\hat\lambda} h(\hat\lambda e^{\int_0^{t} r(s)\,ds}\psi_x(t), \hat{C}(t))\,dt \biggm|\F_{\sigma_{\varepsilon}}\biggr].\nonumber 
\end{eqnarray}
Here the first inequality is the second condition of \eqref{FOC2game} for $\tilde C$, $\tilde\lambda$; the strict inequality holds by the strict monotonicity of $h$ in $c$ and $\hat C(t)>\tilde C(t)$ on $[\sigma_\varepsilon,\tau_\varepsilon)$, which is non-empty on $\{\sigma_\varepsilon<T\}$; the equality is from the second condition of \eqref{FOC2game} for $\tilde C$, $\tilde\lambda$ being binding at $\tau=\tau_\varepsilon$ on $\{\tau_\varepsilon<T\}$ (since $\tau_\varepsilon$ is a point of increase of some $\tilde C^j$, $j=1,\dots,n$), using iterated expectations; the next inequality is the second condition of \eqref{FOC2game} for $\hat C$, $\hat\lambda$; the last inequality is from $\lambda \mapsto \frac{1}{\lambda}h(\lambda\psi,c)$ being strictly decreasing (see again Lemma \ref{fn:h_c<0} in the appendix) and $\hat\lambda\geq\tilde\lambda$. To avoid a contradiction in \eqref{long}, we must have $\sigma_\varepsilon=T$ $\Pb$-a.s.\ for all $\varepsilon>0$, and thus $\hat C\leq\tilde C$ on $[0,T)$, $\Pb$-a.s.

Notice now that with $u_{xc}\geq 0$, $g(\psi,c)$ is nondecreasing in $c$ (and strictly decreasing in $\psi$). Thus, given the ordering $\hat C\leq\tilde C$ ($P\otimes dt$-a.e.) and $\hat\lambda\geq\tilde\lambda$, we obtain $\hat x^i(t)=g\bigl(\hat\lambda \psi_x(t),\hat C(t)\bigr)\leq\tilde x^i(t)=g\bigl(\tilde\lambda \psi_x(t),\tilde C(t)\bigr)$. We may assume that the first hypothesized equilibrium is symmetric, $\hat C^i=\frac1n\hat C$, by \emph{Step 3}. In that equilibrium, any player, say $i=1$, now has a feasible deviation by following $\frac1n\tilde C\geq\frac1n\hat C$, which costs $\Exp\bigl[\int_0^T\psi_c(t)\,d\tilde C^i(t) \bigr]\leq w/n-\Exp\bigl[\int_0^T \psi_x(t) \tilde{x}^i(t) dt\bigr]\leq w/n-\Exp\bigl[\int_0^T \psi_x(t) \hat{x}^i(t) dt\bigr]$. If $\frac1n\tilde C>\frac1n\hat C$ with positive measure, the former would be strictly better since $u^i$ is strictly increasing in $c$, so we must have $\tilde C=\hat C$ $\Pb\otimes dt$ a.e.\ if the latter is indeed an equilibrium. By right-continuity, then even $\hat C(t)=\tilde C(t)$ for all $t\in[0,T)$, $\Pb$-a.s.
\vspace{0.25cm}

\emph{Step 5.}\,\,\,To complete the proof of necessity in the second claim of the theorem, it remains to prove uniqueness of the private consumption and of the costs of individual contributions in equilibrium. Consider two Nash equilibria \((\hat{x}^{1},\dots,\hat{x}^{n},\hat{C}^{1},\dots,\hat{C}^{n})\) and \((\tilde{x}^{1},\dots,\tilde{x}^{n},\tilde{C}^{1},\dots,\tilde{C}^{n})\) with respective symmetric Lagrange multipliers $\hat\lambda^1=\cdots=\hat\lambda^n=:\hat\lambda$ and $\tilde\lambda^1=\cdots=\tilde\lambda^n=:\tilde\lambda$ by \emph{Step 2}. Let $\hat\lambda\geq\tilde\lambda$ without loss of generality. By \emph{Step 3}, $\hat C=\tilde C$ $\Pb\otimes dt$ a.e. Then the corresponding private consumptions are $\hat x^i(t)=g\bigl(\hat \lambda\psi_x(t), \hat C(t)\bigr)\leq\tilde x^i(t)=g\bigl(\tilde\lambda\psi_x(t), \tilde C(t)\bigr)$ $\Pb\otimes dt$ a.e.\ by the monotonicity of $g(\psi,c)$ in $\psi$, strictly if $\hat\lambda>\tilde\lambda$. However, by \emph{Step 1}, $\sum_i\Exp\bigl[\int_0^T \psi_x(t) \hat{x}^i(t) dt\bigr]=w-\Exp\bigl[\int_0^T\psi_c(t)\,d\hat C(t) \bigr]=w-\Exp\bigl[\int_0^T\psi_c(t)\,d\tilde C(t) \bigr]=\sum_i\Exp\bigl[\int_0^T \psi_x(t) \tilde{x}^i(t) dt\bigr]$, so we must have $\hat\lambda=\tilde\lambda$ and thus $\hat x^i=\tilde x^i$, $\Pb\otimes dt$ a.e. Then, again by \emph{Step 1}, $\Exp\bigl[\int_0^T\psi_c(t)\,d\hat C^i(t) \bigr]=w^i-\Exp\bigl[\int_0^T \psi_x(t) \hat{x}^i(t) dt\bigr]=w^i-\Exp\bigl[\int_0^T \psi_x(t) \tilde{x}^i(t) dt\bigr]=\Exp\bigl[\int_0^T\psi_c(t)\,d\tilde C^i(t) \bigr]$.
\vspace{0.25cm}

Finally, in the next two steps we establish existence of a symmetric equilibrium. First, we argue that one can take the solution of an auxiliary single-agent problem with any given budget $w'$ to construct a symmetric $n$-agent equilibrium with some resulting budget $\bar w$. To find a symmetric  equilibrium for a given budget $w$, we argue afterwards that one can chose the initial single-agent budget $w'$ to match $\bar w=w$.\vspace{0.25cm}

\emph{Step 6.}\,\,\,$n$ is our given number of symmetric agents. Consider now the auxiliary single-agent problem with $n'=1$ and utility function $u$ of the symmetric agents and arbitrary budget $w'>0$. This problem has a solution $(x'_*,C'_*)$ by Theorem \ref{SPexistence} (with $\gamma^1=1$) and a Lagrange multiplier $\lambda'_*$ by necessity in Proposition \ref{prop1SP} such that the first order conditions \eqref{FOC2} are satisfied. The latter differ from the first order conditions \eqref{FOC2game} in a symmetric $n$-agent Nash equilibrium only by the budget constraint, where we have $dC'_*$ for the single agent and $d\hat C/n$ in the symmetric equilibrium (the factor $1/n$ does not affect the respective third condition). If we now define $\hat C^i:=C'_*/n$ (so that $\hat C \equiv C'_*$) and $\hat\lambda^i:=\lambda'_*$ (whence $\hat x^i=x'_*$ via the function $g$), then we have a symmetric solution to the equilibrium first order conditions \eqref{FOC2game} with individual budgets
\begin{align*}
w^i \equiv \frac{\bar w}{n}:={}&\Exp\biggl[\int_0^T \psi_x(t) g(\hat{\lambda} e^{\int_0^{t} r(s)\,ds}\psi_x(t), \hat{C}(t))\, dt + \int_0^T  \psi_c(t)\frac1n d\hat{C}(t)\biggr].
\end{align*} 
and hence a symmetric equilibrium for the aggregate budget $\bar w$. Notice that simple algebra implies
\begin{equation}
\Exp\biggl[\int_0^T  \psi_c(t)\frac1n d\hat{C}(t)\biggr] = \Exp\biggl[\int_0^T \psi_c(t) d\hat{C}(t)\biggr] - \frac{n-1}{n}\Exp\biggl[\int_0^T \psi_c(t) d\hat{C}(t)\biggr]
\end{equation}
and therefore we can rewrite
\begin{equation}
\label{rewritewbar}
\frac{\bar w}{n}=w'-\frac{n-1}{n}\Exp\biggl[\int_0^T \psi_c(t)\,d\hat{C}(t)\biggr],
\end{equation}
upon recalling $\hat C \equiv C'_*$ and using the budget constraint of the auxiliary single agent with initial wealth $w'$.
\vspace{0.25cm}

\emph{Step 7.}\,\,\,We now have to find an initial $w'$ such that the construction in \emph{Step 6} yields $\bar w$ equal to our given $w$. First of all, by noticing that $0 \leq \Exp[\int_0^T \psi_c(t)\,d\hat{C}(t)] \leq w'$ since $\hat C \equiv C'_*$, it follows from \eqref{rewritewbar} that $\bar w\in[w',nw']$ and therefore the candidates $w'=w/n$ and $w'=w$ yield symmetric equilibria with budgets enclosing $w$. Thus it remains to show that $\bar w$ depends continuously on $w'$ to complete the proof. 

As the first two elements of any solution $(x'_*,C'_*,\lambda'_*)$ of the first order conditions \eqref{FOC2} for the auxiliary single-agent problem are $P\otimes dt$-essentially unique, they can be understood as a function of $w'$, and thus also $\bar w$ is a function of $w'$ by \eqref{rewritewbar} with $\hat C \equiv C'_*$ and $E[\psi_c(T)\,\Delta\hat C(T)]=0$ by \emph{Step 1}.

Now we show that $(x'_*,C'_*)$ is non-decreasing in $w'$. Therefore consider a second single-agent budget $w''>0$ and corresponding solution $(x''_*,C''_*,\lambda''_*)$ of \eqref{FOC2}. Suppose $\lambda'_*\leq\lambda''_*$. We have shown in the first part of \emph{Step 4}~-- which applies also to the particular case $n'=1$, when Nash equilibrium and social planner solution become the same concept, and further to two equilibria with different budgets, since only the last two conditions of \eqref{FOC2game} were used~--  that under the assumption $u_{xc}\geq 0$, we must then have $C''_*\leq C'_*$, $P\otimes dt$-a.e. Further, as $g(\psi,c)$ is strictly decreasing in $\psi$ and now $g_c(\psi,c)=-u_{xc}(g(\psi,c),c))/u_{xx}(g(\psi,c),c)\geq 0$, also $x''_*(t)=g\bigl(\lambda''_* \psi_x(t), C''_*(t)\bigr)\leq g\bigl(\lambda'_* \psi_x(t), C'_*(t)\bigr)=x'_*(t)$ $P\otimes dt$-a.e., strictly if $\lambda'_*<\lambda''_*$. Due to the monotonicity of $u$ in $(x,c)$, now the maximized utilities given in \eqref{SPproblem} must satisfy $V_{SP}(w'')\leq V_{SP}(w')$, strictly if $\lambda'_*<\lambda''_*$. By switching roles, we get the opposite inequalities throughout if $\lambda'_*\geq\lambda''_*$.

On the other hand, if $w''\geq w'$, then
$\mathcal{B}_{w'}\subset\mathcal{B}_{w''}$ and thus $V_{SP}(w'')\geq V_{SP}(w')$. In combination with the contraposition of the previously established implication $\lambda'_*<\lambda''_*\Rightarrow V_{SP}(w'')< V_{SP}(w')$, we obtain $w''\geq w'\Rightarrow V_{SP}(w'')\geq V_{SP}(w')\Rightarrow\lambda'_*\geq\lambda''_*\Rightarrow C'_*\leq C''_*$ and $x'_*\leq x''_*$, $P\otimes dt$-a.e.

Now consider a monotone sequence of single-agent budgets $\{w'_k\}_{k\in \mathbb{N}}$ converging to some $w'>0$, with the associated monotone sequences $\{C'_{*k}\}_{k\in \mathbb{N}}$ and $\{x'_{*k}\}_{k\in \mathbb{N}}$ (and $\{\lambda'_{*k}\}_{k\in \mathbb{N}}$). Then
\begin{align*}
\lim_{k\to\infty}w'_k&=\lim_{k\to\infty}\Exp\biggl[\int_0^T \psi_x(t) x'_{*k}(t)\, dt + \int_0^T  \psi_c(t) dC'_{*k}(t)\biggr]\\
&=w'=\Exp\biggl[\int_0^T \psi_x(t) x'_*(t)\, dt + \int_0^T  \psi_c(t) dC'_*(t)\biggr].
\end{align*}
Here both integrals are non-decreasing in $k$ in expectation, the first clearly by monotonicity of $\{x'_{*k}\}_{k\in \mathbb{N}}$. Further, $C'_{*k}\leq C'_{*k+1}$ $\Pb\otimes dt$ a.e.\ must also imply $\delta_k:=\Exp[\int_0^T \psi_c(t) dC'_{*k}(t)]-\Exp[\int_0^T  \psi_c(t) dC'_{*k+1}(t)]\leq 0$, because if $\delta_k>0$, then $(x'_k,C'_k):=(x'_{*k}+\delta_k/\Exp[\int_0^T \psi_x(t) dt],C'_{*k+1})\in\mathcal{B}_{w'_k}$ would yield $U^{1}(x'_k,C'_k)>U^{1}(x'_{*k},C'_{*k})$ since $u$ is strictly increasing in $(x,c)$. By monotonicity, both integrals must converge in expectation. Thus the associated sequence
\[
\frac{\bar w_k-\bar w}{n}=w'_k-w'-\frac{n-1}{n}\Exp\biggl[\int_0^T \psi_c(t) \Big(dC'_{*k}(t)-dC'_*(t)\Big)\biggr]
\]
from the construction in \emph{Step 6} must vanish as $k\to\infty$. In summary, $\bar w$ depends continuously on $w'$, such that we can indeed find an initial $w'$ yielding $\bar w=w$.\vspace{0.25cm}
\end{proof}

Since the first order conditions in a symmetric equilibrium are the same for all agents and of a similar form as those for the social planner, we can ``solve'' them in form of a signal process $\hat l$ as in Theorem \ref{SPoptpolicy}. The proof uses analogous arguments as that of Theorem \ref{SPoptpolicy} and is therefore omitted.

\begin{thm}
\label{SymmetricNashoptimalpolicy}
Let Assumption \ref{asm2} hold. Also, assume that all the agents have the same initial wealth $w^i=w/n>0$ and the same utility function $u^i\equiv u$ satisfying Assumption \ref{asm1}
. Then there is a Nash equilibrium of the public good contribution game \eqref{Vi} given by 
\begin{equation}
\label{optimalNashsymm}
\left\{
\begin{array}{ll}
\hat C^i(t)=\sup_{0 \leq u \leq t}\hat l(u) \vee 0, \\ \\
\hat x^i(t)=g\bigl(\hat{\lambda} \psi_x(t), \hat C(t)\bigr)
\end{array}
\right.
\end{equation}
for all $i=1,\dots,n$ if $\hat l$ is an optional, upper right-continuous process solving the stochastic backward equation
\begin{equation}
\label{Nashbacksymm}
\Exp\biggl[\int_{\tau}^T e^{-\int_0^t r(s)\,ds} h\Big(\hat{\lambda} e^{\int_0^{t} r(u)du }\psi_x(t), n \sup_{\tau \leq u \leq t} l(u)\Big)\,dt \biggm|\F_{\tau}\biggr] = \hat{\lambda}\psi_c(\tau)\mathds{1}_{\{\tau < T\}} 
\end{equation}
for any stopping time $\tau \in [0,T]$ with $\hat l_T=0$, \(\Pb\)-a.s., for a Lagrange multiplier $\hat{\lambda}>0$ such that $\Psi(\hat x^i,\hat C^i)=w/n$.
\end{thm}
\begin{rem}
Notice that if the Inada conditions from Proposition \ref{prop:existenceback} hold for $h$, then we have also necessity in Theorem \ref{SymmetricNashoptimalpolicy} for the symmetric equilibrium identified in Theorem \ref{thm:eqlunique}.
\end{rem}

\section{The Free Rider Effect}
\label{freerider}

In economics, the free rider problem occurs when those who benefit from resources, goods, or services do not pay (sufficiently) for them, which results in either an under-provision of those goods or services, or in an overuse or degradation of a common property resource. The free rider problem is common among public goods, because of their \textsl{non-excludability} -- once provided it is for everybody -- and \textsl{non-rivalry} -- the consumption of the good by an agent does not reduce the amount available to others. We refer to \cite{CornesSandler96} or \cite{Laffont88} for more details.

The next proposition shows that in a symmetric setting, the total expenditure for the public good is higher if one follows the symmetric social planner's optimal policy than the symmetric Nash equilibrium.
\begin{prop}
\label{prop:expenditure}
Under the assumptions of Theorem \ref{thm:eqlunique}, let $(x_*^i, C_*^i) \equiv (\frac{1}{n} x_*, \frac{1}{n}C_*)$, $i=1,\dots,n$, represent the symmetric social planner solution of Theorem \ref{SPoptpolicy} with $\gamma^i=1/n$, $n>1$, and $(\hat x^i, \hat C^i) \equiv (\frac{1}{n} \hat x, \frac{1}{n}\hat C)$, $i=1,\dots,n$, represent the symmetric Nash equilibrium of Theorem \ref{thm:eqlunique}.
If $\Exp[C_*(T)]>0$, then
$$\Exp\biggl[ \int_0^T \psi_c(t)(dC_*(t) - d\hat{C}(t))\biggr] > 0;$$
that is, the total expenditure for the public good is higher for the social planner.
\end{prop}
\begin{proof}
If $\Exp[C_*(T)]>0$, then the social planner solution cannot be a Nash equilibrium by comparing the first order conditions \eqref{FOC2} and \eqref{FOC2game} (here with $\gamma^i=1/n$, $n>1$). Indeed, if it were, then by the budget constraint we would need to have $\lambda_*/\gamma^i=\hat\lambda^i$, whence the respective third conditions could not be reconciled if $dC_*>0$ with positive measure. 
By uniqueness of the social planner optimal policy (cf.\ Theorem \ref{SPexistence}), the symmetric social planner solution $(\frac{1}{n}x_*, \frac{1}{n}C_*)$ now gives everybody a higher expected utility than the symmetric equilibrium $(\frac{1}{n}\hat x, \frac{1}{n}\hat C)$. Hence
\begin{eqnarray}
\label{expenditure1}
&& 0 <  \Exp\biggl[\int_0^T e^{-\int_0^t r(s)\,ds} \left[u\Big(\frac{1}{n}x_*(t),C_*(t)\Big)- u\Big(\frac{1}{n}\hat{x}(t),\hat C(t)\Big)\right] dt\biggr] \nonumber \\
&& \leq \Exp\biggl[\int_0^T e^{-\int_0^t r(s)\,ds} u_x\Big(\frac{1}{n}\hat{x}(t), \hat{C}(t)\Big)\frac{1}{n}(x_*(t) - \hat{x}(t))\,dt\biggr] \nonumber \\
&& +\, \Exp\biggl[\int_{0}^T e^{-\int_0^t r(s)\,ds} u_c\Big(\frac{1}{n}\hat{x}(t), \hat{C}(t)\Big)(C_*(t) - \hat{C}(t))\,dt\biggr]   \\
&& = \Exp\biggl[\int_0^T e^{-\int_0^t r(s)\,ds} u_x\Big(\frac{1}{n}\hat{x}(t), \hat{C}(t)\Big)\frac{1}{n}(x_*(t) - \hat{x}(t))\,dt\biggr] \nonumber \\
&& +\, \Exp\biggl[\int_0^T \Exp\biggl[\int_{t}^T e^{-\int_0^s r(u)\,du}  u_c\Big(\frac{1}{n}\hat{x}(s), \hat{C}(s)\Big)\,ds\biggm|\F_{t}\biggr](dC_*(t) - d\hat{C}(t))\biggr] \nonumber 
\end{eqnarray}
where concavity of $u$ has been used to derive the second inequality, whereas an application of Fubini's Theorem and Theorem 1.33 in \cite{Jacod79} justify the last step (see also Remark \ref{rem:supergrad}).
Employing now the first order conditions \eqref{FOCgame} in the last two expected values in \eqref{expenditure1} and noting that $\hat\lambda^i>0$ we get
\begin{equation}
\label{expenditure2}
\Exp\biggl[\int_0^T \frac{1}{n}\psi_x(t)(x_*(t) - \hat{x}(t)) dt + \int_0^T \psi_c(t)(dC_*(t) - d\hat{C}(t))\biggr] > 0.
\end{equation}
On the other hand, subtracting the budget constraint associated to the symmetric social planner solution (i.e.\ the first of \eqref{FOC}) and the aggregated budget constraint associated to the symmetric Nash equilibrium (i.e.\ the sum over all the agents of the first in \eqref{FOCgame}) one easily obtains
\begin{equation}
\label{expenditure3}
\Exp\biggl[\int_0^T \psi_x(t)(x_*(t) - \hat{x}(t))dt + \int_0^T \psi_c(t)(dC_*(t) - d\hat{C}(t))\biggr] =0.
\end{equation}
Therefore, by \eqref{expenditure2} and \eqref{expenditure3} one finds
$$\Exp\biggl[\int_0^T \psi_x(t)(x_*(t) - \hat{x}(t))dt\biggr] < 0 \qquad \text{and} \qquad \Exp\biggl[ \int_0^T \psi_c(t)(dC_*(t) - d\hat{C}(t))\biggr] > 0;$$
that is, the total expenditure for the public good is higher for the social planner.
\end{proof}

In order to find if our model exhibits a (dynamic) free rider effect one should be able to compare the efficient social planner solution $C_*$ with a Nash equilibrium outcome of the public good contribution game, $\hat{C}$, so to show $C_* \geq \hat{C}$ $\Pb \otimes dt$-a.e. In the generality of our model that seems a difficult task, even in a symmetric setting in which all the economic agents share the same utility function and the same wealth. Indeed, a comparison of the symmetric Nash equilibrium of \eqref{optimalNashsymm} with the social planner solution \eqref{optimalSP} seems possible only if a ranking between the Lagrange multipliers $\lambda_*$ and $\hat{\lambda}$ is known. Unfortunately, we have not been able to determine such a ranking in the general framework.

We have thus specified the setting, so to perform direct calculations and explicitly evaluate the free rider effect. Our chosen model satisfies the following assumption which shall hold throughout the remainder of this section.
\begin{asm}
\label{ass:freerider}
\indent\par
\begin{enumerate}
	\item $u^i(x,c) = \frac{x^{\alpha}c^{\beta}}{\alpha + \beta}$, $i=1,\dots,n$, for some $\alpha, \beta \in (0,1)$ such that $\alpha + \beta < 1$;
	\item $w^i = w_o$ for all $i=1,\dots,n$;
	\item $\psi_x(t)=e^{-rt} \mathcal{E}_x(t) = e^{-rt + Z_x(t) -\pi_x(1)t}$ and $\psi_c(t)=e^{-rt} \mathcal{E}_c(t) = e^{-rt + Z_c(t) -\pi_c(1)t}$ for some independent L\'evy processes $Z_x$ and $Z_c$ such that the Laplace transforms 
\[
\pi_x(\xi):=\log\Exp[e^{\xi Z_x(1)}] \quad\text{and}\quad \pi_c(\xi):=\log\Exp[e^{\xi Z_c(1)}],
\] 
are well defined for all $\xi \in \mathbb{R}$. 
\end{enumerate}
\end{asm}
We refer the reader to \cite{Bertoin96}, among others, for a detailed introduction to L\'evy processes. Notice that uncertainty in our model is still much more general than the Brownian one commonly assumed in the literature. Indeed we drop the assumption of normally distributed increments, while keeping a convenient Markovian structure. Our setup also covers the case of jump processes, like the Poisson process, or that of jump-diffusion processes, as well as the deterministic case. 

The next lemma will be useful in the following
\begin{lem}
\label{newLevy}
Let Assumption \ref{ass:freerider} hold and recall the measure $\tilde{\Pb}_c$ such that $\frac{d\tilde{\Pb}_c}{d\Pb}=\mathcal{E}_c(T)$ on $\F_T$. The processes 
$$\hat{Z}(t):=-\frac{\beta}{1-\alpha - \beta}(Z_c(t)-\pi_c(1)t) - \frac{\alpha\beta}{(1-\alpha)(1-\alpha - \beta)} (Z_x(t)-\pi_x(1)t)$$ and $$\widetilde{Z}(t):=-\frac{\alpha}{1-\alpha-\beta}(Z_x(t)-\pi_x(1) t) - \frac{1-\alpha}{1-\alpha-\beta}(Z_c(t)-\pi_c(1)t)$$
are L\'evy. Moreover, 
\begin{enumerate}
	\item the Laplace exponent under $\Pb$ of $\hat{Z}$ at $\xi=2$ is given by
\begin{eqnarray*}
\hat{\pi}(2)&\hspace{-0.25cm}:= \hspace{-0.25cm}&\log\Exp[e^{2\hat{Z}(1)}] \\
&\hspace{-0.25cm} = \hspace{-0.25cm}& \frac{2}{1-\alpha-\beta}\Big[\beta\pi_c(1) + \frac{\alpha\beta}{1-\alpha}\pi_x(1)\Big] + \pi_c\left(-\frac{2\beta}{1-\alpha-\beta}\right) + \pi_x\left(-\frac{2\alpha\beta}{(1-\alpha)(1-\alpha-\beta)}\right);
\end{eqnarray*}
\item the Laplace exponent under $\tilde{\Pb}_c$ of $\widetilde{Z}$ at $\xi=1$ is given by 
\begin{eqnarray*}
\widetilde{\pi}(1)&\hspace{-0.25cm}:= \hspace{-0.25cm}&\log\widetilde{\Exp}_c[e^{\widetilde{Z}(1)}]\\
&\hspace{-0.25cm}= \hspace{-0.25cm}&\frac{1}{1-\alpha-\beta}\Big[\beta\pi_c(1) + \alpha\pi_x(1)\Big] + \pi_c\left(-\frac{\beta}{1-\alpha-\beta}\right) + \pi_x\left(-\frac{\alpha}{1-\alpha-\beta}\right).
\end{eqnarray*}
\end{enumerate}
\end{lem}
\begin{proof}
The processes $\hat{Z}$ and $\widetilde{Z}$ are L\'evy because they are linear combinations of independent L\'evy processes.
The expression of the Laplace exponent $\hat{\pi}(2)$ can be easily obtained exploiting independence of $Z_x$ and $Z_c$. On the other hand, recalling that $\frac{d\tilde{\Pb}_c}{d\Pb}=\mathcal{E}_c(1)$ on $\F_1$ and using again independence of $Z_x$ and $Z_c$, $\widetilde{\pi}(1)$ follows.
\end{proof}

We now make the following assumption to ensure finiteness of relevant quantities in the following.
\begin{asm}
\label{ass:rbig}
$r>\max\Big\{0,\frac{\alpha}{1-\alpha}\pi_x(1), \frac{\alpha}{1-\alpha}\pi_x(1) + \pi_x(-\frac{2\alpha}{1-\alpha}), \hat{\pi}(2) + \frac{\alpha}{1-\alpha}\pi_x(1), \widetilde{\pi}(1)\Big\}$.
\end{asm}
Assumptions \ref{ass:freerider} and \ref{ass:rbig} will be kept throughout this section. In the next two propositions we explicitly solve the social planner problem and we find the explicit form of the symmetric Nash equilibrium of Theorem \ref{SymmetricNashoptimalpolicy}.

\begin{prop}
\label{solsymmetric}
Assume that the social planner weights are $\gamma^i= \frac{1}{n}$ for every $i=1,\dots,n$ and define the processes
\begin{equation}
\label{procgamma}
\gamma(t):=\frac{1}{A}\Big[\left(\frac{\alpha + \beta}{\alpha}\right)\mathcal{E}_x(t) \inf_{0 \leq s \leq t} \left(\mathcal{E}_c^{\frac{\beta(1 - \alpha)}{1 - \alpha - \beta}}(s)\mathcal{E}_x^{\frac{\alpha\beta}{1 - \alpha - \beta}}(s)\right)\Big]^{-\frac{1}{1 - \alpha}},
\end{equation}
\begin{equation}
\label{proctheta}
\theta(t):=\sup_{0 \leq s \leq t}\left(\mathcal{E}_c^{-\frac{1 - \alpha}{1 - \alpha - \beta}}(s) \mathcal{E}_x^{-\frac{\alpha}{1 - \alpha - \beta}}(s) \right),
\end{equation}
and the constants
\begin{equation}
\label{lzero}
l_0:=\frac{nw_o}{\displaystyle \Exp\biggl[\int_0^{\infty} \psi_x(t) \gamma(t) dt  + \int_{0}^{\infty}\psi_c(t)d\theta(t)\biggr]}
\end{equation}
and
\begin{equation}
\label{A}
A:=\Exp\biggl[\int_{0}^{\infty}\delta e^{- r u} \inf_{0 \leq s \leq u} \left(\mathcal{E}_c(s) \mathcal{E}_x^{-\frac{\alpha}{1 - \alpha}}(u-s)\right)du\biggr]
\end{equation}
with $\delta:=\frac{\beta}{\alpha}\left(\frac{\alpha + \beta}{\alpha}\right)^{\frac{1}{\alpha-1}}$.

Then the social planner's optimal solution is such that
\begin{equation}
\label{SPpublic}
C_{*}(t)= l_0 \theta(t)
\end{equation}
and
\begin{equation}
\label{SPprivate}
x^{i}_{*}(t) = \frac{1}{n} l_0 \gamma(t), \quad i=1,\dots,n,
\end{equation}
with
\begin{equation*}
\lambda_* = \frac{1}{n^{\alpha}} A^{1-\alpha} l_0^{\alpha + \beta - 1}.
\end{equation*}
\end{prop}
\begin{proof}
By Theorem \ref{SPoptpolicy}, to find the social planner's optimal policy it suffices to solve the backward equation \eqref{SPback}.

Recall that $h^i(\psi,c)=u^i_c(g^i(\psi,c),c)$, where $g^i(\cdot, c)$ is the inverse of $u^i_x(\cdot, c)$. For any $\lambda > 0$, simple algebra leads to $h^i(\frac{\lambda}{\gamma^i} e^{rt}\psi_x(t), C(t)) = \delta( n \lambda \mathcal{E}_x(t))^{\frac{\alpha}{\alpha-1}}C^{\frac{\alpha + \beta -1}{1 - \alpha}}(t)$ with $\delta:=\frac{\beta}{\alpha}\left(\frac{\alpha + \beta}{\alpha}\right)^{\frac{1}{\alpha-1}}$. 
Set $C_{*}(t)= \sup_{0\leq s \leq t}l^{*}(s) \vee 0$ for some progressively measurable process $l^{*}$ to be found and then (\ref{SPback}) becomes
\begin{equation*}
\Exp\biggl[\int_{\tau}^{\infty} \delta e^{-rs} (n\lambda \mathcal{E}_x(s))^{\frac{\alpha}{\alpha-1}} \Big( \sup_{\tau \leq u \leq s} l^{*}(u)\Big)^{\frac{\alpha + \beta -1}{1 - \alpha}} ds \biggm|\F_{\tau}\biggr] = \lambda e^{-r\tau} \mathcal{E}_c(\tau),
\end{equation*} 
i.e.,
\begin{equation}
\label{backSPexample}
\Exp\biggl[\int_{0}^{\infty} \delta e^{-ru} (n\lambda)^{\frac{\alpha}{\alpha-1}} \frac{{\mathcal{E}_x}^{\frac{\alpha}{\alpha-1}}(u+\tau)}{\mathcal{E}_c(\tau)}  \inf_{0 \leq s \leq u}\Big( {l^{*}}^{\frac{\alpha + \beta -1}{1 - \alpha}}(s+\tau)\Big) du \biggm|\F_{\tau}\biggr] = \lambda.
\end{equation}
Make now the \textsl{ansatz} $l^{*}(t):=l_0{\mathcal{E}_{c}}^{\frac{1 -\alpha}{\alpha + \beta -1}}(t){\mathcal{E}_x}^{\frac{\alpha}{\alpha + \beta -1}}(t)$ for some constant $l_0$, and use independence and stationarity of L\'evy increments to rewrite (\ref{backSPexample}) as
\begin{equation*}
\frac{1}{n^{\frac{\alpha}{1 - \alpha}}} l_0^{\frac{\alpha + \beta - 1}{1 - \alpha}} \Exp\biggl[\int_{0}^{\infty} \delta e^{-ru} \inf_{0 \leq s \leq u}\left(\mathcal{E}_c(s)\mathcal{E}_x^{\frac{\alpha}{\alpha -1}}(u-s)\right) du \biggr] = \lambda^{\frac{1}{1-\alpha}}.
\end{equation*}
Set now $A:= \Exp[\int_{0}^{\infty}\delta e^{- r u} \inf_{0 \leq s \leq u} (\mathcal{E}_c(s) \mathcal{E}_x^{-\frac{\alpha}{1 - \alpha}}(u-s))du]$ and notice that it is finite. 
Indeed 
\begin{eqnarray*}
&& \Exp\bigg[\int_{0}^{\infty}\delta e^{- r u} \inf_{0 \leq s \leq u} (\mathcal{E}_c(s) \mathcal{E}_x^{-\frac{\alpha}{1 - \alpha}}(u-s))du\bigg] \leq \Exp\bigg[\int_{0}^{\infty}\delta e^{- r u} \mathcal{E}_c(u) du\bigg] \\
&& = \int_{0}^{\infty}\delta e^{- r u} \Exp\big[\mathcal{E}_c(u)\big] du = \frac{\delta}{r},
\end{eqnarray*}
where the second step follows from Tonelli's theorem, whereas the last one by the fact that $\mathcal{E}_c$ is a martingale (cf.\ Assumption \ref{asm2}).

Then, by solving the previous equation for $\lambda$ one easily obtains
\begin{equation*}
\lambda = \frac{1}{n^{\alpha}} A^{1-\alpha} l_0^{\alpha + \beta - 1}=:\lambda_*.
\end{equation*}
On the other hand, $x^i_{*}(t) = [ n\lambda_* \left(\frac{\alpha + \beta}{\alpha}\right) \mathcal{E}_x(t) C_{*}^{-\beta}(t)]^{\frac{1}{\alpha - 1}}$ by (\ref{binding2}) and therefore
\begin{equation*}
x^i_{*}(t) = (n\lambda_*)^{-\frac{1}{1-\alpha}}\biggl[\left(\frac{\alpha + \beta}{\alpha}\right) \mathcal{E}_x(t) l_0^{-\beta}\inf_{0 \leq s \leq t}\left(\mathcal{E}_c^{\frac{\beta(1-\alpha)}{1-\alpha - \beta}}(s)\mathcal{E}_x^{\frac{\alpha\beta}{1-\alpha - \beta}}(s)\right)\biggr]^{-\frac{1}{1-\alpha}};
\end{equation*}
that is,
\begin{equation}
\label{xistar2}
x^i_{*}(t) = \frac{1}{n} l_0 \gamma(t)
\end{equation}
with $\gamma(t)$ as in (\ref{procgamma}).

To determine $l_0$ we make use of the budget constraint $\mathbb{E}[\int_{0}^{\infty}\psi_x(t) x_{*}(t) dt + \int_{0}^{\infty}\psi_c(t) dC_{*}(t)] = nw_o$. In fact, recalling that $x_{*} := \sum_{i=1}^n x^i_{*}$, from (\ref{xistar2}) we find
\begin{equation}
\label{findinglzero}
l_0 \Exp\biggl[\int_0^{\infty} \psi_x(t) \gamma(t) dt  + \int_{0}^{\infty}\psi_c(t)d\theta(t)\biggr] = nw_o,
\end{equation}
since $C_{*}(t) = \sup_{0 \leq s \leq t}l^{*}(s) = l_0 \sup_{0 \leq s \leq t}(\mathcal{E}_c^{-\frac{1 - \alpha}{1 - \alpha - \beta}}(s) \mathcal{E}_x^{-\frac{\alpha}{1 - \alpha - \beta}}(s)) = l_0 \theta(t)$ with $\theta(t)$ as in (\ref{proctheta}) and 
if $\Exp[\int_0^{\infty} \psi_x(t) \gamma(t) dt + \int_{0}^{\infty}\psi_c(t)d\theta(t)] < \infty$. Then, by solving \eqref{findinglzero} for $l_0$, \eqref{lzero} follows.
\smallskip

To conclude the proof it thus remains to show that $\Exp[\int_0^{\infty} \psi_x(t) \gamma(t) dt + \int_{0}^{\infty}\psi_c(t)d\theta(t)] < \infty$. 

We start showing $\Exp[\int_0^{\infty} \psi_x(t) \gamma(t) dt]< \infty$. Define the L\'evy process $\hat{Z}(t):=-\frac{\beta}{1-\alpha - \beta}(Z_c(t)-\pi_c(1)t) - \frac{\alpha\beta}{(1-\alpha)(1-\alpha - \beta)} (Z_x(t)-\pi_x(1)t)$ (cf.\ Lemma \ref{newLevy}) and denote by $\tau_{\xi}$ an independent random time, exponentially distributed with parameter $\xi:=r -\frac{\alpha}{1-\alpha}\pi_x(1)$. Noticing that $\xi>0$ by Assumption \ref{ass:rbig}, one has 
\begin{eqnarray*}
&& \hspace{-0.25cm} A \left(\frac{\alpha+\beta}{\alpha}\right)^{\frac{1}{1-\alpha}}\Exp\bigg[\int_0^{\infty} \psi_x(t) \gamma(t) dt\bigg] = \Exp\bigg[\int_0^{\infty} e^{-rt}\mathcal{E}_x^{-\frac{\alpha}{1-\alpha}}(t)\sup_{0 \leq s \leq t}\left(\mathcal{E}_c^{\frac{\beta(1-\alpha)}{1-\alpha - \beta}}(s)\mathcal{E}_x^{\frac{\alpha\beta}{1-\alpha - \beta}}(s)\right)^{-\frac{1}{1-\alpha}}dt\bigg]\nonumber\\
&& \hspace{-0.25cm}= \xi^{-1}\Exp\bigg[e^{-\frac{\alpha}{1-\alpha}Z_x(\tau_{\xi})}e^{\sup_{0\leq s \leq \tau_{\xi}}\hat{Z}(s)}\bigg] \leq \xi^{-1}\Exp\bigg[e^{-\frac{2\alpha}{1-\alpha}Z_x(\tau_{\xi})}\bigg]^{\frac{1}{2}}\Exp\bigg[e^{2\sup_{0\leq s \leq \tau_{\xi}}\hat{Z}(s)}\bigg]^{\frac{1}{2}},
\end{eqnarray*}
where H\"older's inequality has been used in the last step. Denoting by $\hat{\pi}$ the Laplace exponent of $\hat{Z}$ under $\Pb$, by the Wiener-Hopf factorization (see, e.g., Chapter VI.2 in \cite{Bertoin96}) we have $\Exp[e^{2\sup_{0\leq s \leq \tau_{\xi}}\hat{Z}(s)}]<\infty$ and $\Exp[e^{-\frac{2\alpha}{1-\alpha}Z_x(\tau_{\xi})}]<\infty$ because $\xi>\hat{\pi}(2) \vee \pi_x(-\frac{2\alpha}{1-\alpha})$ by definition of $\xi$ and Assumption \ref{ass:rbig}.
\smallskip

\noindent We now prove $\Exp[\int_{0}^{\infty}\psi_c(t)d\theta(t)] < \infty$. Fix $0< T<\infty$ and recall (cf.\ proof of Theorem \ref{SPexistence}) the measure $\tilde{\Pb}_c$ such that $\frac{d\tilde{\Pb}_c}{d\Pb}=\mathcal{E}_c(T)$ on $\F_T$. Then a change of measure and integration by parts lead to
$$\Exp\bigg[\int_{0}^{T}\psi_c(t)d\theta(t)\bigg] = \tilde{\Exp}_c\bigg[\int_0^{T} e^{-rt}d\theta(t)\bigg] =\tilde{\Exp}_c\bigg[e^{-rT}\theta(T) -1 + \int_0^T re^{-rt}\theta(t)dt\bigg],$$
where we have used that $\theta(0)=1$ due to \eqref{proctheta}.
Therefore, taking limits as $T\uparrow \infty$ and invoking the monotone convergence theorem we obtain 
$$\Exp\bigg[\int_{0}^{\infty}\psi_c(t)d\theta(t)\bigg] = \tilde{\Exp}_c\Big[ \lim_{T\uparrow\infty} e^{-rT}\theta(T)\Big] + \tilde{\Exp}_c\bigg[\int_0^{\infty} re^{-rt}\theta(t)dt\bigg] - 1.$$
Hence $\tilde{\Exp}_c[\int_0^{\infty} re^{-rt}\theta(t)dt]< \infty$ is necessary to have $\Exp[\int_{0}^{\infty}\psi_c(t)d\theta(t)]<\infty$. It is actually also sufficient since it also implies $\lim_{T\uparrow\infty} e^{-rT}\theta(T)=0$ $\tilde{\Pb}_c$-a.s.\ as it is shown, e.g., in \cite{BankRiedel01}, proof of Lemma 4.9-(i). But now
\begin{equation}
\label{finite}
\tilde{\Exp}_c\bigg[\int_0^{\infty} re^{-rt}\theta(t)dt\bigg] = \tilde{\Exp}_c\bigg[e^{\sup_{0\leq s \leq \tau_r}\widetilde{Z}(s)}\bigg],
\end{equation}
where we have defined the L\'evy process (cf.\ Lemma \ref{newLevy}) $\widetilde{Z}(t):=-\frac{\alpha}{1-\alpha-\beta}(Z_x(t)-\pi_x(1) t) - \frac{1-\alpha}{1-\alpha-\beta}(Z_c(t)-\pi_c(1)t)$ and where $\tau_r$ is an independent random time exponentially distributed with parameter $r$. By Wiener-Hopf factorization we conclude that the right hand-side of \eqref{finite} is finite if $r>\widetilde{\pi}(1)$ where $\widetilde{\pi}$ is the Laplace exponent of $\widetilde{Z}$ under $\tilde{\Pb}_c$ (cf.\ Lemma \ref{newLevy}). The proof is now complete due to Assumption \ref{ass:rbig}.
\end{proof}

We now explicitly solve the best reply problems (\ref{Vi}). The proof of the following result employs arguments similar to those used for the proof of Proposition \ref{solsymmetric}; it is relegated to Appendix \ref{AppProofs}, Section \ref{proofgameexample}, for the sake of completeness.

\begin{prop}
\label{gamesolsymmetric}
Define the processes
\begin{equation}
\label{gamegamma}
\gamma(t):=\frac{1}{A}\Big[\left(\frac{\alpha + \beta}{\alpha}\right)\mathcal{E}_x(t) \inf_{0 \leq s \leq t} \left(\mathcal{E}_c^{\frac{\beta(1 - \alpha)}{1 - \alpha - \beta}}(s)\mathcal{E}_x^{\frac{\alpha\beta}{1 - \alpha - \beta}}(s)\right)\Big]^{-\frac{1}{1 - \alpha}},
\end{equation}
\begin{equation}
\label{gametheta}
\theta(t):=\sup_{0 \leq s \leq t}\left(\mathcal{E}_c^{-\frac{1 - \alpha}{1 - \alpha - \beta}}(s) \mathcal{E}_x^{-\frac{\alpha}{1 - \alpha - \beta}}(s) \right),
\end{equation}
and the constants
\begin{equation}
\label{kappa}
\kappa:=\frac{w_o}{\displaystyle \Exp\biggl[\int_0^{\infty} \psi_x(t) \gamma(t) dt  + \frac{1}{n}\int_{0}^{\infty}\psi_c(t)d\theta(t)\biggr]}
\end{equation}
and
\begin{equation}
\label{gameA}
A:=\Exp\biggl[\int_{0}^{\infty}\delta e^{- r u} \inf_{0 \leq s \leq u} \left(\mathcal{E}_c(s) \mathcal{E}_x^{-\frac{\alpha}{1 - \alpha}}(u-s)\right)du\biggr]
\end{equation}
with $\delta:=\frac{\beta}{\alpha}\left(\frac{\alpha + \beta}{\alpha}\right)^{\frac{1}{\alpha-1}}$.

Then the symmetric Nash equilibrium of game (\ref{Vi}) is given by
\begin{equation}
\label{gamepublic}
\hat{C}^i(t)= \frac{\kappa}{n}\theta(t), \quad i=1,\dots,n,
\end{equation}
\begin{equation}
\label{gameprivate}
\hat{x}^{i}(t) = \kappa \gamma(t), \quad i=1,\dots,n,
\end{equation}
with
\begin{equation*}
\hat{\lambda}^i =  A^{1-\alpha} \kappa^{\alpha + \beta - 1}, \quad i=1,\dots,n.
\end{equation*}
\end{prop}

Thanks to the results of Propositions \ref{solsymmetric} and \ref{gamesolsymmetric}, we are now able to explicitly evaluate the free rider effect for our symmetric economy with Cobb-Douglas utilities and L\'evy uncertainty.
Let $x^{i}_{*}$ be the optimal private consumption in the social planner's problem (cf.\ (\ref{SPprivate})), and let $\hat{x}^{i}$ denote the Nash equilibrium private consumption (cf.\ (\ref{gameprivate})). Then one has 
\begin{align*}
\displaystyle x^{i}_{*}(t) &= \frac{w_o\gamma(t)}{\displaystyle \Exp\biggl[\int_0^{\infty} \psi_x(t) \gamma(t) dt  + \int_{0}^{\infty}\psi_c(t)d\theta(t)\biggr]} \nonumber \\
&\leq \displaystyle \frac{w_o\gamma(t)}{\displaystyle \Exp\biggl[\int_{0}^{\infty} \psi_x(t) \gamma(t) dt  + \frac{1}{n}\int_0^{\infty}\psi_c(t)d\theta(t)\biggr]} = \hat{x}^{i}(t), \nonumber
\end{align*}
with equality for $n=1$. It follows that in a strategic context each agent spends more for the private consumption than what would be suggested by the social planner.
On the other hand, we have $\kappa \leq l_0$ (with $\kappa$ as in (\ref{kappa}), $l_0$ as in (\ref{lzero}) and equality if $n=1$) which implies that 
$$C_*(t) = l_0 \theta(t) \geq \kappa \theta(t) = \hat{C}(t), \quad \Pb\text{-a.s.}\,\,\forall\; t\geq 0;$$
that is, the social planner's optimal cumulative contribution into the public good (\ref{SPpublic}) is bigger than the corresponding Nash equilibrium counterpart (\ref{gamepublic}). Our model thus exhibits a free rider effect.

\subsection{Free Rider Effect: The Role of Uncertainty}
\label{free-rider-uncertainty}

The evaluation of the free rider effect can be made even more explicit in a Black-Scholes setting and with the public good taken as a num\'{e}raire. 

\begin{prop}
\label{freeridereffectprop}
Let $C_{*}$ be the optimal aggregated public good contribution for the social planner problem (cf.\ (\ref{SPpublic})) and let $\hat{C}$ denote its symmetric Nash-equilibrium value (cf.\ (\ref{gamepublic})). Assume $\psi_c(t)=e^{-rt}$ and $\psi_x(t)=e^{-rt}\mathcal{E}_{x}(t) \equiv e^{-rt + \sigma W(t)}$, $\sigma > 0$, for a one-dimensional Brownian motion $W$ and for some $r$ such that $\sqrt{2r} > \frac{\sigma \alpha}{1-\alpha-\beta}$.\footnote{Notice that the martingale property of Assumption \ref{asm2} is without loss of generality in this case, one just has to correct \(r\) by $\frac{1}{2}\sigma^2$, i.e.\ by the Laplace exponents of $\sigma \frac{W(t)}{t}$.} Then, for any $n \geq 1$ one has
\begin{equation}
\label{freeridereffectirr}
\frac{\hat{C}(t)}{C_{*}(t)}=\frac{\kappa}{l_0}=\frac{\alpha + \beta}{n\alpha + \beta} \leq 1,
\end{equation}
where $\kappa$ and $l_0$ are as in (\ref{kappa}) and (\ref{lzero}), respectively.
\end{prop}

\begin{proof}
From (\ref{SPpublic}) and (\ref{gamepublic}) it easily follows that 
\begin{eqnarray}
\label{irrfreeridingratio}
\frac{\hat{C}(t)}{C_{*}(t)} = \frac{\kappa}{l_0} &\hspace{-0.25cm}=\hspace{-0.25cm}&\frac{\Exp\biggl[\displaystyle \int_{0}^{\infty}e^{-rt}\mathcal{E}_{x}(t)\gamma(t)dt + \int_{0}^{\infty}e^{-rt}d\theta(t)\biggr]}{\Exp\biggl[\displaystyle n\int_{0}^{\infty}e^{-rt}\mathcal{E}_{x}(t)\gamma(t)dt + \int_{0}^{\infty}e^{-rt}d\theta(t)\biggr]}\nonumber \\
&\hspace{-0.25cm}=\hspace{-0.25cm}&\frac{\Exp\biggl[\displaystyle \int_{0}^{\infty}e^{-rt}\mathcal{E}_{x}(t)\gamma(t)dt + r\int_{0}^{\infty}e^{-rt}\theta(t)dt\biggr]}{\Exp\biggl[\displaystyle n\int_{0}^{\infty}e^{-rt}\mathcal{E}_{x}(t)\gamma(t)dt + r\int_{0}^{\infty}e^{-rt}\theta(t)dt\biggr]},
\end{eqnarray}
with $\gamma(t)$ and $\theta(t)$ as in (\ref{gamegamma}) and (\ref{gametheta}), respectively.
Then, in order to obtain (\ref{freeridereffectirr}), we need to evaluate
$$\Exp\biggl[\int_{0}^{\infty}e^{-rt}\mathcal{E}_{x}(t)\gamma(t)dt\biggr] \qquad\,\, \text{and} \qquad\,\, \Exp\biggl[\int_{0}^{\infty}re^{-rt}\theta(t)dt\biggr].$$
We have
\begin{eqnarray}
\label{integraltheta}
&&\Exp\biggl[\int_{0}^{\infty}re^{-rt}\theta(t)dt\biggr] = \Exp\biggl[\int_{0}^{\infty}re^{-rt} \sup_{0 \leq s \leq t} \mathcal{E}_{x}^{-\frac{\alpha}{1-\alpha-\beta}}(s)\, dt\biggr] =\Exp\Big[e^{-\frac{\sigma \alpha}{1-\alpha-\beta}\inf_{0 \leq s \leq \tau_r}W(s)}\Big] \nonumber\\
&& = \Exp\Big[e^{\frac{\sigma \alpha}{1-\alpha-\beta}\sup_{0 \leq s \leq \tau_r}(\tilde{W}(s))}\Big] =\frac{\sqrt{2r}}{\sqrt{2r} - \frac{\sigma \alpha}{1-\alpha-\beta}},
\end{eqnarray}
where $\tilde{W}:=-W$, $\tau_r$ is an independent exponentially distributed random time and where the last equality follows from $\sup_{0 \leq s \leq \tau_r}(\tilde{W}(s)) \sim Exp{(\sqrt{2r})}$ (cf., e.g., \cite{Bertoin96}, Chapter VII or \cite{BorodinSalminen}).
On the other hand, recall $\gamma$ as in (\ref{gamegamma}) and exploit that $W(\tau_r) - \sup_{0\leq u \leq \tau_r}W(u)$ is independent of $\sup_{0\leq u \leq \tau_r}W(u)$ (see, e.g., Theorem VI.5(i) in \cite{Bertoin96}) and that $W(\tau_r) - \sup_{0\leq u \leq \tau_r}W(u)$ has the same distribution as $\inf_{0\leq u \leq \tau_r}W(u)$ (Duality Theorem) to find
\begin{eqnarray*}
&&\Exp\biggl[\int_{0}^{\infty}e^{-rt}\mathcal{E}_{x}(t)\gamma(t)dt\biggr] 
=\frac{1}{rA}\left(\frac{\alpha + \beta}{\alpha}\right)^{-\frac{1}{1-\alpha}} \Exp\biggl[\int_{0}^{\infty}re^{-rt} e^{-\frac{\sigma \alpha}{1-\alpha}W(t)}
e^{-\frac{\sigma \alpha\beta}{(1-\alpha)(1-\alpha-\beta)}\inf_{0\leq u \leq t}W(u)}dt\biggr] \\
&&=\frac{1}{rA}\left(\frac{\alpha + \beta}{\alpha}\right)^{-\frac{1}{1-\alpha}} \Exp\biggl[e^{\frac{\sigma \alpha}{1-\alpha}[\tilde{W}(\tau_r)) - \sup_{0\leq u \leq \tau_r}\tilde{W}(u)]} e^{\frac{\sigma \alpha}{1-\alpha-\beta}\sup_{0\leq u \leq \tau_r}\tilde{W}(u)} \biggr] \\
&&=\frac{1}{rA}\left(\frac{\alpha + \beta}{\alpha}\right)^{-\frac{1}{1-\alpha}} \Exp\biggl[e^{\frac{\sigma \alpha}{1-\alpha}\inf_{0\leq u \leq \tau_r}\tilde{W}(u)}\biggr]\Exp\biggl[e^{\frac{\sigma \alpha}{1-\alpha-\beta}\sup_{0\leq u \leq \tau_r}\tilde{W}(u)} \biggr] \\
&&=\frac{1}{rA}\left(\frac{\alpha + \beta}{\alpha}\right)^{-\frac{1}{1-\alpha}} \left[\frac{\sqrt{2r}}{\sqrt{2r} + \frac{\sigma \alpha}{1-\alpha}}\right]\left[\frac{\sqrt{2r}}{\sqrt{2r} - \frac{\sigma \alpha}{1-\alpha - \beta}}\right],
\end{eqnarray*}
where we have used once more $\sup_{0 \leq s \leq \tau_r}\tilde{W}(s) \sim Exp{(\sqrt{2r})}$.
Also, setting $\delta:=\frac{\beta}{\alpha}\left(\frac{\alpha + \beta}{\alpha}\right)^{-\frac{1}{1-\alpha}}$ we have 
\begin{eqnarray*}
A & \hspace{-0.25cm}  = \hspace{-0.25cm}  & \Exp\biggl[\int_0^{\infty}\delta\,e^{-rt}\inf_{0 \leq s \leq t}\mathcal{E}_{x}^{-\frac{\alpha}{1-\alpha}}(t-s)\,dt\biggr] = \frac{\delta}{r}\Exp\biggl[e^{-\frac{\sigma \alpha}{1-\alpha} \sup_{0 \leq s \leq \tau_r}W(\tau_r -s)}\biggr] \\
& \hspace{-0.25cm} = \hspace{-0.25cm} & \frac{\delta}{r}\Exp\biggl[e^{-\frac{\sigma \alpha}{1-\alpha} \sup_{0 \leq s' \leq \tau_r}W(s')}\biggr] = \frac{\delta}{r}\left[\frac{\sqrt{2r}}{\sqrt{2r} + \frac{\sigma \alpha}{1-\alpha}}\right].
\end{eqnarray*}
Therefore
\begin{eqnarray}
\label{integralgamma2}
& & \Exp\biggl[\int_{0}^{\infty}e^{-rt}\mathcal{E}_{x}(t)\gamma(t)dt\biggr] = \frac{\alpha}{\beta}\left[\frac{\sqrt{2r}}{\sqrt{2r} - \frac{\sigma \alpha}{1-\alpha - \beta}} \right].
\end{eqnarray}
Finally, by plugging (\ref{integraltheta}) and (\ref{integralgamma2}) into (\ref{irrfreeridingratio}), some simple algebra leads to (\ref{freeridereffectirr}). 
\end{proof}

We observe that the ratio $\hat{C}/C_{*}$, the underprovision of the public good due to free-riding, does not depend on $\sigma$, the volatility of the Brownian motion $W$. Thus, in our model
\begin{cor}
The degree of free-riding does not depend on the level of uncertainty.
\end{cor} 
This seems to be in contrast to the finding of other models in the economic literature in which it is shown that uncertainty may play some role in the free rider effect (cf.\ \cite{Austen-Smith80}, \cite{EichbergerKelsey99} and \cite{Wang10}, among others).

\subsection{Free Rider Effect: The Role of Irreversibility of Public Good Contributions}
\label{free-rider-irreversibility}

As in Proposition \ref{freeridereffectprop} assume $\psi_c(t)=e^{-rt}$ and $\psi_x(t)=e^{-rt}\mathcal{E}_{x}(t) \equiv e^{-rt + \sigma W(t)}$, $\sigma > 0$, for a one-dimensional Brownian motion $W$ and for some $r$ such that $\sqrt{2r} > \frac{\sigma \alpha}{1-\alpha-\beta}$. We now study the role played by irreversibility of the public good contributions (i.e.\ by the fact that these are monotone controls, possibly singular with respect to Lebesgue measure as functions of time) in the free rider effect. Given initial wealths $w^i=w_o$, $i=1,\dots,n$, define the nonempty, convex sets
\begin{equation}
\label{Swi}
\begin{split}
\mathcal{S}_{w_o}:=\biggl\{&(x^i, c^i) : \Omega \times [0,T] \mapsto \mathbb{R}_{+}^2\, \text{ optional, s.t.\ } \\ 
& \Exp\biggl[\int_0^T \psi_x(t) x^i(t) dt + r\int_0^T  \psi_c(t) c^i(t) dt\biggr] \leq w_o \biggr\}. 
\end{split}
\end{equation}
and
\begin{equation}
\label{Sw}
\begin{split}
\mathcal{S}_{nw_o}:=\biggl\{&(\underline{x}, \underline{c}) : \Omega \times [0,T] \mapsto \mathbb{R}_{+}^{2n}\, \text{ optional, s.t.\ } \\ 
&\sum_{i=1}^n \Exp\biggl[\int_0^T \psi_x(t) x^i(t) dt + r\int_0^T  \psi_c(t) c^i(t) dt\biggr] \leq nw_o \biggr\}.
\end{split}
\end{equation}
Notice that, differently to \eqref{Bwi} and \eqref{Bw}, in these two sets we do not require anymore the public good provisions to be monotone processes, possibly singular with respect to the Lebesgue measure (as functions of time). We have indeed here that the cumulative investment made up to time $t$, $C^i(t)$, satisfies $dC^i(t)=c^i(t) dt$.

Consider now the social planner problem and the public good contribution game (cf.\ Section \ref{SocialPlanner} and Section \ref{game}) when, however, the social planner picks her investment strategies from the set $\mathcal{S}_{nw_o}$ and agents from the set $\mathcal{S}_{w_o}$; that is, they face the optimal control problems
$$v_{SP}:=\sup_{(\underline{x}, \underline{c}) \in \mathcal{S}_{nw_o}} \sum_{i=1}^n \gamma^i U^{i}(x^i,c^{i};c^{-i})$$
and 
$$v^{i}(c^{-i}):= \sup_{(x^i,c^{i})\in\mathcal{S}_{w_o}}U^{i}(x^i,c^{i};c^{-i}), \qquad i=1,\dots,n.$$
Finally, we denote by $(\underline{x}_*, \underline{c}_*)$ the optimal policy of the social planner and, similarly to Definition \eqref{precNash}, we say 
\begin{defn}
\label{precNash-bis}
\((\hat{x}^{1},\dots,\hat{x}^{n},\hat{c}^{1},\dots,\hat{c}^{n})\) is a \emph{Nash equilibrium} if for all \(i\in\{1,\dots,n\}\), \((\hat{x}^{i},\hat{c}^{i})\in\mathcal{S}_{w_o}\) and \(U^{i}(\hat{x}^{i}, \hat{c}^{i},\hat{c}^{-i})=v^{i}(\hat{c}^{-i})\).
\end{defn}

Economically, taking public good contributions in the set $\mathcal{S}_{w_o}$ or $\mathcal{S}_{nw_o}$, we are assuming perfect reversibility of $C$; i.e., each agent can adjust contribution in the public good freely at every point of time.

\begin{prop}
Under the same assumptions of Proposition \ref{freeridereffectprop}, one has
\begin{equation}
\label{freeridingrev}
\frac{\hat{c}(t)}{c_{*}(t)} = \frac{\alpha + \beta}{n\alpha + \beta}
\end{equation}
for any $n \geq 1$.
\end{prop}
\begin{proof}
Employing arguments similar to those in the proof of Proposition \ref{prop1SP} (but now only with classical controls) one can show that the first-order conditions for optimality in the social planner's problem read $\Pb\otimes dt$-a.e.\footnote{Notice that such a result is in line with the well known finding from the economic literature (see, e.g., \cite{Jorgenson63}) that under perfect reversibility the optimal investment criterion is to equate the marginal operating profit with the user cost of capital.}
\begin{equation}
\label{FOCSPrev}
\left\{
\begin{array}{ll}
\displaystyle \frac{\alpha}{\alpha + \beta}(x_{*}^{i})^{\alpha-1}(t)c_{*}^{\beta}(t) = \lambda_{*} n\mathcal{E}_{x}(t),  \\ \\
\displaystyle \frac{\beta}{\alpha + \beta}(x_{*}^i)^{\alpha}(t)c_{*}^{\beta-1}(t) = \lambda_{*} r, \\ \\
\displaystyle \Exp\biggl[\int_0^{\infty}\psi_x(t)\sum_{i=1}^n x^i_{*}(t)\,dt + r\int_{0}^{\infty}e^{-rt}c_{*}(t)\,dt\biggr] = nw_o, 
\end{array}
\right.
\end{equation}
which are always binding because the Cobb-Douglas utility satisfies Inada conditions in both variables. On the other hand, for the Nash equilibrium they are $\Pb\otimes dt$-a.e.
\begin{equation}
\label{FOCSPnash}
\left\{
\begin{array}{ll}
\displaystyle \frac{\alpha}{\alpha + \beta}(\hat{x}^i)^{\alpha-1}(t)\hat{c}^{\beta}(t) = \hat{\lambda} \mathcal{E}_{x}(t),  \\ \\
\displaystyle \frac{\beta}{\alpha + \beta}(\hat{x}^i)^{\alpha}(t)\hat{c}^{\beta-1}(t) = \hat{\lambda} r, \\ \\
\displaystyle \Exp\biggl[\int_0^{\infty}\psi_x(t) \hat{x}^i(t)\,dt + r\int_{0}^{\infty}e^{-rt}\frac{1}{n}\hat{c}(t)\,dt\biggr] = w_o.
\end{array}
\right.
\end{equation}
By solving systems \eqref{FOCSPrev} and \eqref{FOCSPnash} one easily obtains
\begin{equation*}
\left\{
\begin{array}{ll}
\displaystyle x_{*}^i(t) = \left(\frac{r\alpha}{\beta}\right)\Big[\lambda_{*}\left(\frac{\alpha + \beta}{\alpha}\right)\left(\frac{r\alpha}{\beta}\right)^{1-\alpha}\Big]^{-\frac{1}{1-\alpha-\beta}} (n\mathcal{E}_{x}(t))^{-\frac{(1-\beta)}{1-\alpha-\beta}},  \\ \\
\displaystyle c_{*}(t)= \Big[\lambda_{*}\left(\frac{\alpha + \beta}{\alpha}\right)\left(\frac{r\alpha}{\beta}\right)^{1-\alpha}\Big]^{-\frac{1}{1-\alpha-\beta}} (n\mathcal{E}_{x}(t))^{-\frac{\alpha}{1-\alpha-\beta}}, \\ \\
\displaystyle \lambda_{*}^{-\frac{1}{1-\alpha-\beta}}=\frac{nw_o}{rn^{-\frac{\alpha}{1-\alpha-\beta}}\left(\frac{\alpha+\beta}{\beta}\right)\Big[ \left(\frac{r\alpha}{\beta}\right)^{1-\alpha}\left(\frac{\alpha + \beta}{\alpha}\right)\Big]^{-\frac{1}{1-\alpha-\beta}}\Exp\biggl[\displaystyle \int_0^{\infty}e^{-rt}\mathcal{E}_{x}^{-\frac{\alpha}{1-\alpha-\beta}}(t)\,dt\biggr]}, 
\end{array}
\right.
\end{equation*}
and
\begin{equation*}
\left\{
\begin{array}{ll}
\displaystyle \hat{x}^i(t) = \left(\frac{r\alpha}{\beta}\right)\Big[\hat{\lambda}\left(\frac{\alpha + \beta}{\alpha}\right)\left(\frac{r\alpha}{\beta}\right)^{1-\alpha}\Big]^{-\frac{1}{1-\alpha-\beta}} \mathcal{E}_{x}^{-\frac{(1-\beta)}{1-\alpha-\beta}}(t),  \\ \\
\displaystyle \hat{c}(t)= \Big[\hat{\lambda}\left(\frac{\alpha + \beta}{\alpha}\right)\left(\frac{r\alpha}{\beta}\right)^{1-\alpha}\Big]^{-\frac{1}{1-\alpha-\beta}} \mathcal{E}_{x}^{-\frac{\alpha}{1-\alpha-\beta}}(t), \\ \\
\displaystyle \hat{\lambda}^{-\frac{1}{1-\alpha-\beta}}=\frac{nw_o}{rn^{-\frac{\alpha}{1-\alpha-\beta}}\left(\frac{n\alpha+\beta}{n\beta}\right)\Big[ \left(\frac{r\alpha}{\beta}\right)^{1-\alpha}\left(\frac{\alpha + \beta}{\alpha}\right)^{-\frac{1}{1-\alpha-\beta}}\Big]^{-\frac{1}{1-\alpha-\beta}}\Exp\biggl[\displaystyle \int_0^{\infty}e^{-rt}\mathcal{E}_{x}^{-\frac{\alpha}{1-\alpha-\beta}}(t)\,dt\biggr]},
\end{array}
\right.
\end{equation*}
with $\Exp\biggl[\displaystyle \int_0^{\infty}e^{-rt}\mathcal{E}_{x}^{-\frac{\alpha}{1-\alpha-\beta}}(t)\,dt\biggr] < \infty$ since $\sqrt{2r} > \frac{\sigma \alpha}{1-\alpha-\beta}$.
Then (\ref{freeridingrev}) easily follows. 
\end{proof}

Comparing \eqref{freeridereffectirr} and \eqref{freeridingrev} one is then left with the following result.
\begin{cor}
One has
\begin{equation*}
\frac{\hat{C}(t)}{C_{*}(t)} = \frac{\hat{c}(t)}{c_{*}(t)} \leq 1.
\end{equation*}
That is, irreversibility of the public good contributions does not influence the degree of free-riding.
\end{cor}

In conclusion, we have shown that in our model, for a symmetric economy with Cobb-Douglas utilities, the degree of underprovision of the public good due to free-riding does not depend on irreversibility of the public good contributions or the level of uncertainty, when the latter is given by an exogenous one-dimensional Brownian motion. This interesting conclusion sheds new light on the old economic problem of public good contribution showing that irreversibility and uncertainty not necessarily mitigate the degree of free riding.
\bigskip

\noindent \textbf{Aknowledgments.}\,\,The authors wish to thank two anonymous referees for their pertinent and useful comments.


\appendix
\section{Some Proofs and Technical Results}
\label{AppProofs}
\renewcommand{\theequation}{A-\arabic{equation}}

\subsection{On the Proof of Proposition \ref{prop1SP}}
\label{proofFOCSprop}

In this section we prove Proposition \ref{prop1SP}. The proof is a generalization of Theorem 3.2 in \cite{BankRiedel01} to the case of a multivariate optimal consumption problem with both monotone and classical absolutely continuous controls. Sufficiency easily follows from concavity of the utility functions $u^i$, $i=1,\dots,n$. On the other hand, the next Lemma accomplishes the proof of the necessity part. Necessity is proved by linearizing the original problem (\ref{SPproblem}) around its optimal solution $(\underline{x}_{*}, \underline{C}_{*})$, by showing that $(\underline{x}_{*}, \underline{C}_{*})$ solves the linearized problem as well and that it satisfies some flat-off conditions as those of (\ref{FOC}). 

Recall the notation $x:=\sum_{i=1}^n x^i$ and $C:=\sum_{i=1}^n C^i$. 
\begin{lem}
\label{NecessityLemma1}
Let Assumptions \ref{asm1} and \ref{asm2} hold and $(\underline{x}_{*}, \underline{C}_{*}) \in \mathcal{B}_w$ be optimal for problem (\ref{SPproblem}) and set 
\beq
\label{Psistardef}
\Psi_{*}(t):= \Exp\biggl[\int_{t}^T e^{-\int_0^s r(u)\,du} \sum_{i=1}^n \gamma^i\,u^{i}_c(x^i_{*}(s), C_{*}(s))\,ds\Big|\F_{t}\biggr]. 
\eeq
Then $(\underline{x}_{*}, \underline{C}_{*})$
\begin{enumerate}
	\item solves the linear optimization problem
\begin{equation}
\label{linearproblem}
\sup_{(\underline{x}, \underline{C}) \in \mathcal{B}_w} \Exp\biggl[\int_0^T e^{-\int_0^t r(s)ds} \sum_{i=1}^n \gamma^i u^i_x(x^i_{*}(t),C_{*}(t)) x^i(t)dt + \int_0^T\Psi_{*}(t)dC(t)\biggr];
\end{equation}
\item satisfies
\begin{equation}
\label{flatoff}
\left\{
\begin{array}{ll}
\displaystyle \Big(e^{-\int_0^t r(s)ds} \gamma^i u^i_x(x^i_{*}(t),C_{*}(t)) - M\psi_x(t)\Big)x^i_*(t) = 0 \quad \,\,\Pb\otimes dt\text{-a.e.},\, i=1,\dots,n, \\ \\
\displaystyle \Exp\biggl[\int_0^T \Big(\Psi_{*}(t) - M \psi_c(t)\Big)dC_{*}(t)\biggr] = 0,
\end{array}
\right.
\end{equation}
with 
\begin{align}\label{moltLagrange}
M:={}&(\Pb\otimes dt)\text{-}\esssup \Biggl[\max\Biggl\{\frac{e^{-\int_0^t r(s)ds} \gamma^i u^i_x(x^i_{*}(t),C_{*}(t))}{\psi_x(t)};\; i=1,\dots,n \Biggr\}\Biggr] \nonumber\\[6pt]
\vee\; &\Pb\text{-}\esssup \Biggl[\sup_{t\in [0,T]} \frac{\Psi_{*}(t)}{\psi_c(t)}\Biggr].
\end{align}
\end{enumerate}
\end{lem}

\begin{proof}
The proof splits into two steps. \vspace{0.25cm}

\textsl{Step 1.}\quad Let $(\underline{x}_{*}, \underline{C}_{*}) \in \mathcal{B}_w$ be optimal for problem (\ref{SPproblem}). For $(\underline{x}, \underline{C}) \in \mathcal{B}_w$ and $\epsilon \in [0,1]$, define the admissible strategy $(\underline{x}_{\epsilon}, \underline{C}_{\epsilon})$ with $\underline{x}_{\epsilon}(t):= \epsilon \underline{x}(t) + (1-\epsilon)\underline{x}_{*}(t)$ and such that $C_{\epsilon}(t)=\epsilon C(t) + (1-\epsilon)C_{*}(t)$. Notice that $\underline{x}_{\epsilon}(t)$ and $C_{\epsilon}(t)$ respectively converge to $\underline{x}_{*}(t)$ and $C_{*}(t)$ for all $t \in [0,T]$ a.s.\  when $\epsilon \downarrow 0$. Now, optimality of $(\underline{x}_{*}, \underline{C}_{*})$, Assumption \ref{asm1}.\ref{unint}, concavity of $u^i$ and an application of Fubini's Theorem allow us to write
\begin{eqnarray}
\label{Neclemma1}
0  & \hspace{-0.25cm} \geq \hspace{-0.25cm}& \frac{1}{\epsilon}[U_{SP}(\underline{x}_{\epsilon}, \underline{C}_{\epsilon}) - U_{SP}(\underline{x}_{*}, \underline{C}_{*})] \nonumber \\
& \hspace{-0.25cm}  \geq \hspace{-0.25cm} & \Exp\biggl[\int_{0}^T e^{-\int_0^t r(s)\,ds} \sum_{i=1}^n \gamma^i u^{i}_x(x^i_{\epsilon}(t), C_{\epsilon}(t))(x^i(t) - x^i_{*}(t))\,dt\biggr] \nonumber \\
& & +\, \Exp\biggl[\int_{0}^T e^{-\int_0^t r(s)\,ds} \sum_{i=1}^n \gamma^i u^{i}_c(x^i_{\epsilon}(t), C_{\epsilon}(t))(C(t) - C_{*}(t))\,dt\biggr]   \\
& \hspace{-0.25cm} = \hspace{-0.25cm} & \Exp\biggl[\int_{0}^T e^{-\int_0^t r(s)\,ds} \sum_{i=1}^n \gamma^i u^{i}_x(x^i_{\epsilon}(t), C_{\epsilon}(t))(x^i(t) - x^i_{*}(t))\,dt\biggr] \nonumber \\
& & + \Exp\biggl[\int_0^T \Phi_{\epsilon}(t)(dC(t) - dC_{*}(t))\biggr], \nonumber 
\end{eqnarray}
where $\Phi_{\epsilon}(t):= \int_{t}^T e^{-\int_0^s r(u)\,du} \sum_{i=1}^n \gamma^i\,u^{i}_c(x^i_{\epsilon}(s), C_{\epsilon}(s))\,ds$. 
One has 
\begin{align*}
&\liminf_{\epsilon \downarrow 0} \Exp\biggl[\int_{0}^T e^{-\int_0^t r(s)\,ds} \sum_{i=1}^n \gamma^i u^{i}_x(x^i_{\epsilon}(t), C_{\epsilon}(t))x^i(t) dt \biggr] \\
&\geq \Exp\biggl[\int_{0}^T e^{-\int_0^t r(s)\,ds} \sum_{i=1}^n \gamma^i u^{i}_x(x^i_{*}(t), C_{*}(t))x^i(t) dt \biggr]
\end{align*}
and
\begin{equation*}
\liminf_{\epsilon \downarrow 0} \Exp\biggl[\int_0^T \Phi_{\epsilon}(t) dC(t) \biggr] \geq \Exp\biggl[\int_0^T  \Phi_{*}(t) dC(t) \biggr],
\end{equation*}
with $\Phi_{*}:=\Phi_{0}$, by Fatou's Lemma.
We now claim (and we prove it later) that
\begin{align}
\label{claim1}
&\lim_{\epsilon \downarrow 0} \Exp\biggl[\int_{0}^T e^{-\int_0^t r(s)\,ds} \sum_{i=1}^n \gamma^i u^{i}_x(x^i_{\epsilon}(t), C_{\epsilon}(t))x^i_{*}(t) dt \biggr] \\
&= \Exp\biggl[\int_{0}^T e^{-\int_0^t r(s)\,ds} \sum_{i=1}^n \gamma^i u^{i}_x(x^i_{*}(t), C_{*}(t))x^i_{*}(t) dt \biggr] \notag
\end{align}
and
\begin{equation}
\label{claim2}
\lim_{\epsilon \downarrow 0} \Exp\biggl[\int_0^T \Phi_{\epsilon}(t) dC_{*}(t) \biggr] = \Exp\biggl[\int_0^T \Phi_{*}(t) dC_{*}(t) \biggr].
\end{equation}
Hence from \eqref{Neclemma1}
\begin{align*}
&\Exp\biggl[\int_{0}^T e^{-\int_0^t r(s)\,ds} \sum_{i=1}^n \gamma^i u^{i}_x(x^i_{*}(t), C_{*}(t))x^i(t) \,dt\biggr] + \Exp\biggl[\int_0^T  \Phi_{*}(t)dC(t)\biggr] \\
\leq\; &\Exp\biggl[\int_{0}^T e^{-\int_0^t r(s)\,ds} \sum_{i=1}^n \gamma^i u^{i}_x(x^i_{*}(t), C_{*}(t))x^i_{*}(t) \,dt\biggr] + \Exp\biggl[\int_0^T  \Phi_{*}(t)dC_{*}(t)\biggr]
\end{align*}
and by replacing $\Phi_{*}$ with its optional projection $\Psi_{*}$ as defined in (\ref{Psistardef}) (cf.\ \cite{Jacod79}, Theorem $1.33$) it follows that $(\underline{x}_{*}, \underline{C}_{*})$ is optimal for problem (\ref{linearproblem}) as well.

To conclude the proof we must prove (\ref{claim1}) and (\ref{claim2}). 
To prove (\ref{claim1}) it suffices to show that the family $(\Gamma^1_{\epsilon})_{\epsilon \in [0,\frac{1}{2}]}$ given by
$$\Gamma^1_{\epsilon}(t):= e^{-\int_0^t r(s)\,ds} \sum_{i=1}^n \gamma^i u^{i}_x(x^i_{\epsilon}(t), C_{\epsilon}(t))x^i_{*}(t)$$
is $\Pb \otimes dt$-uniformly integrable.
Concavity of $u^i$ and the fact that $x^{i}_{\epsilon}(t) \geq \frac{1}{2}x^i_{*}(t)$ a.s.\ for $\epsilon \in [0,\frac{1}{2}]$ and every $t \in [0,T]$ lead to
\begin{equation*}
\label{boundgamma1}
\Gamma^1_{\epsilon}(t) \leq 2  e^{-\int_0^t r(s)\,ds} \sum_{i=1}^n \gamma^i u^{i}_x( x^i_{\epsilon}(t),C_{\epsilon}(t))x^i_{\epsilon}(t)  \leq 2  e^{-\int_0^t r(s)\,ds} \sum_{i=1}^n \gamma^i \bigl[u^{i}(x^i_{\epsilon}(t), C_{\epsilon}(t))-u^{i}(0, C_{\epsilon}(t))\bigr], 
\end{equation*}
and the last term in the right-hand side above is $\Pb \otimes dt$-uniformly integrable by Assumption \ref{asm1}.\ref{unint}. Then \eqref{claim1} holds by Vitali's Convergence Theorem.

As for (\ref{claim2}) note that by Fubini's Theorem 
$$\int_0^T \Phi_{\epsilon}(t)dC_{*}(t) = \int_0^T e^{-\int_0^t r(s)\,ds} \sum_{i=1}^n \gamma^i u^{i}_c(x^i_{\epsilon}(t), C_{\epsilon}(t))C_{*}(t)dt.$$
Hence, to have \eqref{claim2} it suffices to show that the family 
$$\Gamma^2_{\epsilon}(t):= e^{-\int_0^t r(s)\,ds} \sum_{i=1}^n \gamma^i u^{i}_c(x^i_{\epsilon}(t), C_{\epsilon}(t))C_{*}(t)$$
is $\Pb \otimes dt$-uniformly integrable, but this follows by employing arguments similar to those used for $(\Gamma^1_{\epsilon})_{\epsilon \in [0,\frac{1}{2}]}$.

\vspace{0.15cm}

\textsl{Step 2.}\quad We now show that the flat-off conditions (\ref{flatoff}) hold for any solution $(\hat{\underline{x}}, \hat{\underline{C}})$ of the linear problem (\ref{linearproblem}). Then, by Step 1, they also hold for $(\underline{x}_{*},\underline{C}_{*})$.

Notice that for every $(\underline{x}, \underline{C}) \in \mathcal{B}_w$ one has
\begin{eqnarray}
\label{primalemma2}
& & \Exp\biggl[\int_0^T \sum_{i=1}^n e^{-\int_0^t r(s)ds} \gamma^i u^i_x(x^i_{*}(t),C_{*}(t)) x^i(t)dt + \int_0^T \Psi_{*}(t) dC(t)\biggr] \nonumber \\
& & \leq M \Exp\biggl[\int_0^T \sum_{i=1}^n \psi_x(t) x^i(t) dt + \int_0^T \psi_c(t)dC(t)\biggr] = Mw
\end{eqnarray}
by definition of $M$ (cf.\ \eqref{moltLagrange}). 
Obviously, if $(\underline{x}, \underline{C})$ satisfies (\ref{flatoff}) we then have equality in (\ref{primalemma2}). On the other hand, if 
\begin{equation}
\label{Mw}
\sup_{(\underline{x}, \underline{C}) \in \mathcal{B}_w} \Exp\biggl[\int_0^T e^{-\int_0^t r(s)ds} \sum_{i=1}^n \gamma^i u^i_x(x^i_{*}(t),C_{*}(t)) x^i(t)dt + \int_0^T \Psi_{*}(t) dC(t)\biggr] = Mw,
\end{equation}
then equality holds through (\ref{primalemma2}) and we obtain (\ref{flatoff}).

It therefore remains to prove (\ref{Mw}).
To this end take $K<M$ and define the investment strategies
$$x^{i}_K(t):=\begin{cases} \alpha & \text{if } e^{-\int_0^t r(s)ds} \gamma^i u^i_x(x^i_{*}(t),C_{*}(t)) \geq K \psi_x(t), \\ 0 & \text{else} \end{cases}
 \quad\text{and}\quad C_K(t):=\alpha\mathds{1}_{[\sigma_K,T]}(t),$$
with the stopping time\[
\sigma_K:= \inf\{t \in [0,T): \Psi_{*}(t) \geq K \psi_c(t)\} \wedge T
\]
and some $\alpha$ such that $\Exp[\int_0^T \sum_{i=1}^n \psi_x(t)x^i_K(t) dt + \int_0^T \psi_c(t)dC_K(t)] =w$. Note that one can find such $\alpha$. Indeed, suppose to the contrary that for $\alpha=1$, $\max\{x^{i}_K(t);i=1,...,n\}=C_K(t)=0$ on $[0,T)$ \nbd{\Pb\otimes dt}a.e.\ for some $K>0$. This would mean $M \leq K$ by the definition of $M$ (cf.\ \eqref{moltLagrange}).

We now have
\begin{eqnarray*}
Mw &\hspace{-0.25cm} \geq  \hspace{-0.25cm}& \sup_{(\underline{x}, \underline{C}) \in \mathcal{B}_w} \Exp\biggl[\int_0^T e^{-\int_0^t r(s)ds} \sum_{i=1}^n \gamma^i u^i_x(x^i_{*}(t),C_{*}(t)) x^i(t)dt + \int_0^T \Psi_{*}(t) dC(t)\biggr] \nonumber \\
&  \hspace{-0.25cm} \geq \hspace{-0.25cm} & \Exp\biggl[\int_0^T e^{-\int_0^t r(s)ds} \sum_{i=1}^n \gamma^i u^i_x(x^i_{*}(t),C_{*}(t)) x^i_K(t)dt + \int_0^T \Psi_{*}(t) dC_K(t)\biggr] \nonumber \\
&\hspace{-0.25cm}  \geq \hspace{-0.25cm}  & K\,\Exp\biggl[\int_0^T \sum_{i=1}^n \psi_x(t)x^i_K(t) dt + \alpha \psi_c(\sigma_K)\mathds{1}_{\{\sigma_K < T\}}\biggr] \nonumber \\
& \hspace{-0.25cm} \geq \hspace{-0.25cm} & K\,\Exp\biggl[\int_0^T \sum_{i=1}^n \psi_x(t)x^i_K(t) dt + \int_0^T \psi_c(t)dC_K(t)\biggr] = Kw,
\end{eqnarray*}
which yields (\ref{Mw}) by letting $K \uparrow M$.
\end{proof}

\noindent We are now able to prove Proposition \ref{prop1SP}. \vspace{0.25cm} 

\noindent \textbf{Proof of Proposition \ref{prop1SP}}\vspace{0.15cm}
\begin{proof} 
Sufficiency follows from concavity of utility function $u^i$, $i=1,\dots,n$, (cf.\ Assumption \ref{asm1}). Indeed, for $(\underline{x}_{*}, \underline{C}_{*}) \in \mathcal{B}_w$ satisfying (\ref{FOC}) and for $(\underline{x}, \underline{C})$ any other admissible policy we may write
\begin{align*}
U_{SP}(\underline{x}_{*}, \underline{C}_{*}) - U_{SP}(\underline{x}, \underline{C}) \geq  &\Exp\biggl[\int_0^T e^{-\int_0^t r(s)ds} \sum_{i=1}^n \gamma^i u^i_{x}(x^i_{*}(t),C_{*}(t))(x^i_{*}(t) - x^i(t))dt\biggr] \\
+ &\Exp\biggl[\int_0^T e^{-\int_0^t r(s)ds} \sum_{i=1}^n \gamma^i u^i_{c}(x^i_{*}(t),C_{*}(t))(C_{*}(t) - C(t)) dt\biggr] \\
= &\Exp\biggl[\int_0^T e^{-\int_0^t r(s)ds} \sum_{i=1}^n \gamma^i u^i_{x}(x^i_{*}(t),C_{*}(t))(x^i_{*}(t) - x^i(t))dt\biggr] \\
+ &\Exp\biggl[\int_0^T \biggl(\int_{t}^T e^{-\int_0^s r(u)\,du} \sum_{i=1}^n \gamma^i u^{i}_c(x^i_{*}(s), C_{*}(s))\,ds\biggr) (dC_{*}(t) - dC(t))\biggr] \\
\geq  &\lambda (w - w) = 0, 
\end{align*}
where (\ref{FOC}) and Fubini's Theorem lead to the second inequality, whereas the last one is implied by the first and the fourth of (\ref{FOC}) and by the budget constraint.
Finally, Lemma \ref{NecessityLemma1} yields the proof of the necessity part.
\end{proof}

\subsection{A Useful Simple Lemma}

\begin{lem}
\label{fn:h_c<0}
Let Assumption \ref{asm1} hold and set \(h^i(\psi,c):=u_c^i(g(\psi,c),c)\) for every \(i=1,\dots,n\), \(\psi,c>0\), where \(g^i(\cdot,c)\) is the inverse of \(u_x^i(\cdot,c)\). Then $c \mapsto h^i(\psi,c)$ is strictly decreasing for any $\psi>0$. Moreover, if also $u_{xc}\geq 0$, then $\psi \mapsto h(\psi,c)$ is nonincreasing for any $c>0$.
\end{lem}
\begin{proof}
$h^i$ is strictly decreasing in $c$ by strict concavity of $u^i$. Indeed, since $u^i_x(g^i(\psi,c),c)=\psi$ is constant in $c$ for all $\psi>0$, by implicit differentiation $g^i_c(\psi,c)=-u^i_{xc}(g^i(\psi,c),c)/u^i_{xx}(g^i(\psi,c),c)$. Then $h^i_c(\psi,c)=u^i_{xc}(g^i(\psi,c),c)g^i_c(\psi,c)+u^i_{cc}(g^i(\psi,c),c)=-\bigl[u^i_{xc}(g^i(\psi,c),c)\bigr]^2/u^i_{xx}(g^i(\psi,c),c)+u^i_{cc}(g^i(\psi,c),c)<0$, as the Hessian of $u^i$ is negative definite by Assumption \ref{asm1}.i.

On the other hand, $h(\psi,c)=u_c(g(\psi,c),c)$ is nonincreasing in $\psi$ if also $u_{xc}\geq 0$, since $g(\cdot,c)$ is strictly decreasing like $u_x(\cdot,c)$, which it is the inverse of.
\end{proof}

\subsection{Proof of Proposition \ref{prop:existenceback}}
\label{proof:prop:existenceback}

\begin{proof}
For any given $\lambda > 0$, set $X(t):=\lambda \psi_c(t)\mathds{1}_{\{t < T\}}$. Such a process  vanishes at $T$, it is of class (D) and lower semicontinuous in expectation by Assumption \ref{asm2}. 
Moreover, we define the atomless, optional random Borel measure $\mu(\omega,dt):= e^{-\int_0^t r(\omega,s) ds}dt$ and the random field
$$f^{i}(\omega,t,\ell):=\begin{cases} h^i(\frac{\lambda}{\gamma^i}e^{\int_0^t r(\omega,s) ds}\psi_x(\omega,t), -\frac{1}{\ell})  & \text{if } \ell<0, \\ -\ell & \text{if } \ell \geq 0, \end{cases}$$
for some $\gamma^i>0$.
Notice that $(\omega,t)\mapsto f^{i}(\omega,t,\ell)$ is progressively measurable and $\Pb\otimes \mu(dt)$ integrable for any given and fixed $\ell \in \mathbb{R}$. Moreover, since $c \mapsto h^i(\psi, c)$ is strictly decreasing (cf.\ Lemma \ref{fn:h_c<0}) and, by assumption, satisfies the Inada conditions 
\beq
\label{app-Inada}
\lim_{c \downarrow 0}h^i(\psi,c)=+\infty, \qquad \lim_{c \uparrow \infty}h^i(\psi,c)=0,
\eeq
then $f^{i}(\omega,t,\cdot)$ is strictly decreasing from $+\infty$ to $-\infty$.
All these properties are clearly inherited by the function $\sum_{i=1}^n \gamma^i h^i(\psi, \cdot)$, being $\gamma^i > 0$, $i=1,\dots,n$. 

Following the arguments, e.g., in the proof of Proposition 3.4 of \cite{Ferrari12} (see also the proof of Theorem 2.4 in \cite{BankRiedel03}), we can apply Theorem $3$ of \cite{BankElKaroui04} to have existence of an optional signal process $l^*$ solving \eqref{SPback}. Such a process is also upper right-continuous and therefore it is unique up to indistinguishability by \cite{BankElKaroui04}, Theorem $1$, and Meyer's optional section theorem (see, e.g., \cite{DellacherieMeyer78}, Theorem IV.86) (cf.\ again \cite{Ferrari12}, proof of Proposition 3.4, or \cite{BankKuechler07}, proof of Theorem $1$).
\end{proof}

\subsection{Proof of Proposition \ref{gamesolsymmetric}}
\label{proofgameexample}
Due to Theorem \ref{SymmetricNashoptimalpolicy}, to find the Nash equilibrium strategy of the public good contribution game \eqref{Vi} in our homogeneous and symmetric setting it suffices to solve backward equation \eqref{Nashbacksymm}.

Recall that $h^i(\psi,c):=u^i_c(g^i(\psi,c),c)$ with $g^i(\cdot, c)$ the inverse of $u^i_x(\cdot, c)$. For any $\lambda^i > 0$, straightforward computations lead to $h^i(\lambda^i e^{rt}\psi_x(t), C(t)) = \delta (\lambda^i \mathcal{E}_x(t))^{\frac{\alpha}{\alpha-1}}C^{\frac{\alpha + \beta -1}{1 - \alpha}}(t)$, with $\delta:=\frac{\beta}{\alpha}\left(\frac{\alpha + \beta}{\alpha}\right)^{\frac{1}{\alpha-1}}$. 
Set $\hat C^i(t)= \sup_{0\leq s \leq t}\hat l(s) \vee 0$ for some progressively measurable process $\hat l$ solving
\begin{equation*}
\Exp\biggl[\int_{\tau}^{\infty} \delta e^{-rs} (\lambda^i \mathcal{E}_x(s))^{\frac{\alpha}{\alpha-1}} \Big( n\sup_{\tau \leq u \leq s} \hat l(u)\Big)^{\frac{\alpha + \beta -1}{1 - \alpha}} ds \biggm|\F_{\tau}\biggr] = \lambda^i e^{-r\tau} \mathcal{E}_c(\tau),
\end{equation*} 
i.e.,
\begin{equation}
\label{backgameexample}
\Exp\biggl[\int_{0}^{\infty} \delta e^{-ru} (\lambda^i)^{\frac{\alpha}{\alpha-1}} \frac{{\mathcal{E}_x}^{\frac{\alpha}{\alpha-1}}(u+\tau)}{\mathcal{E}_c(\tau)}  \inf_{0 \leq s \leq u}\Big( (n{\hat l})^{\frac{\alpha + \beta -1}{1 - \alpha}}(s+\tau)\Big) du \biggm|\F_{\tau}\biggr] = \lambda^i.
\end{equation}
Now take $\hat l(t):=\frac{\kappa}{n}{\mathcal{E}_{c}}^{\frac{1 -\alpha}{\alpha + \beta -1}}(t){\mathcal{E}_x}^{\frac{\alpha}{\alpha + \beta -1}}(t)$ for some constant $\kappa$ and use independence and stationarity of L\'evy increments to rewrite (\ref{backgameexample}) as
\begin{equation}
\label{backgameexample2}
\kappa^{\frac{\alpha + \beta - 1}{1 - \alpha}} \Exp\biggl[\int_{0}^{\infty} \delta e^{-ru} \inf_{0 \leq s \leq u}\left(\mathcal{E}_c(s)\mathcal{E}_x^{\frac{\alpha}{\alpha -1}}(u-s)\right) du \biggr] = {(\lambda^i)}^{\frac{1}{1-\alpha}}.
\end{equation}
We claim (and we discuss later) that (cf.\ \eqref{gameA}) $A:= \Exp[\int_{0}^{\infty}\delta e^{- r u} \inf_{0 \leq s \leq u} \left(\mathcal{E}_c(s) \mathcal{E}_x^{-\frac{\alpha}{1 - \alpha}}(u-s)\right)du]$ is finite. Then by solving (\ref{backgameexample2}) for $\lambda^i$ one obtains
\begin{equation*}
\lambda^i= A^{1-\alpha} \kappa^{\alpha + \beta - 1}=:\hat{\lambda}^i.
\end{equation*}
But now $\hat x^i(t) = [\hat{\lambda}^i \left(\frac{\alpha + \beta}{\alpha}\right) \mathcal{E}_x(t) \hat C^{-\beta}(t)]^{\frac{1}{\alpha - 1}}$, and therefore
\begin{equation*}
\hat x^i(t) = (\hat{\lambda}^i)^{-\frac{1}{1-\alpha}}\biggl[\left(\frac{\alpha + \beta}{\alpha}\right) \mathcal{E}_x(t) {\kappa}^{-\beta}\inf_{0 \leq s \leq t}\left(\mathcal{E}_c^{\frac{\beta(1-\alpha)}{1-\alpha - \beta}}(s)\mathcal{E}_x^{\frac{\alpha\beta}{1-\alpha - \beta}}(s)\right)\biggr]^{-\frac{1}{1-\alpha}};
\end{equation*}
that is,
\begin{equation}
\label{gamexistar2}
\hat x^i(t) = \kappa \gamma(t)
\end{equation}
with $\gamma(t)$ as in (\ref{gamegamma}).

To determine $\kappa$ we use the budget constraint $\mathbb{E}[\int_{0}^{\infty}\psi_x(t) \hat x^i(t) dt + \int_{0}^{\infty}\psi_c(t) d\hat C^i(t)] = w$.
Indeed, by (\ref{gamexistar2}) we have
\begin{equation}
\label{findingkappa}
\kappa \Exp\biggl[\int_0^{\infty} \psi_x(t) \gamma(t) dt  + \frac{1}{n}\int_{0}^{\infty}\psi_c(t)d\theta(t)\biggr] = w,
\end{equation}
since $\hat C^i(t) = \sup_{0 \leq s \leq t}\hat l(s) = \frac{\kappa}{n} \sup_{0 \leq s \leq t}(\mathcal{E}_c^{-\frac{1 - \alpha}{1 - \alpha - \beta}}(s) \mathcal{E}_x^{-\frac{\alpha}{1 - \alpha - \beta}}(s)) = \frac{\kappa}{n} \theta(t)$ with $\theta(t)$ as in (\ref{gametheta}), and if $\Exp[\int_0^{\infty} \psi_x(t) \gamma(t) dt  + \frac{1}{n}\int_{0}^{\infty}\psi_c(t)d\theta(t)]<\infty$.
Now the result follows by solving (\ref{findingkappa}) for $\kappa$.
Finally, arguments as those in the proof of Proposition \ref{solsymmetric} allow to show that under Assumption \ref{ass:freerider}.4. all the quantities above are finite, thus completing the proof.


\end{document}